\documentclass[11pt,reqno]{amsart}
\usepackage[utf8]{inputenc}
\usepackage{amsmath,amssymb,url}
\usepackage{supertabular}
\usepackage{mathrsfs}
\usepackage{mathrsfs}

\newcommand{\CC}{{\mathbb C}}
\newcommand{\RR}{{\mathbb R}}

\newcommand{\ZZ}{{\mathbb{Z}}}

 \DeclareMathOperator{\Prob}{\mathbb{P}}   
 \DeclareMathOperator{\E}{\mathbb{E}}      
 \DeclareMathOperator{\V}{\mathbb{V}}      

 \newcommand{\I}{1\!\!1}                   

 \newcommand{\dd}{d}            
 \newcommand{\ii}{i}

\def\stacksum#1#2{{\stackrel{{\scriptstyle #1}}
{{\scriptstyle #2}}}}

\newcommand{\demi}{{\textstyle{\frac{1}{2}}}}

\newcommand{\hyperg}[4]{\phantom{}_2F_{1}({{#1}},{{#2}};{{#3}};{{#4}})}

\newcommand{\eps}{\varepsilon}
\newcommand{\Cc}{\CC}

\newcommand{\Zz}{\ZZ}

\newcommand{\Rr}{\RR}

\newcommand{\Fp}{\mathbb{F}}

\newcommand{\proba}{\Prob}
\newcommand{\expect}{\E}
\newcommand{\variance}{\V}
\newcommand{\charfun}{\mathbf{1}}

\newcommand{\mods}[1]{\,(\mathrm{mod}\,{#1})}

\newcommand{\ra}{\rightarrow}
\newcommand{\lra}{\longrightarrow}

\newcommand{\fleche}[1]{\stackrel{#1}{\lra}}

\DeclareMathOperator{\Imag}{Im}
\DeclareMathOperator{\Reel}{Re}

\newtheorem{thm}{Theorem}[section]
\newtheorem{cor}[thm]{Corollary}
\newtheorem{lem}[thm]{Lemma}
\newtheorem{prop}[thm]{Proposition}

\newtheorem{conjecture}[thm]{Conjecture}
\theoremstyle{definition}
\newtheorem{defn}[thm]{Definition}
\theoremstyle{remark}
\newtheorem{rem}[thm]{Remark}
\newtheorem{example}[thm]{Example}

\numberwithin{equation}{section}

\begin{document}
\title[Mod-Gaussian convergence]{Mod-Gaussian convergence: new limit
  theorems in probability and number theory}

\author{J. Jacod}
 \address{Institut de math\'ematiques de Jussieu,
 Universit\'e Pierre et Marie Curie, et C.N.R.S. UMR 7586,
 175, rue du Chevaleret,
 F-75013 Paris,
 France}
 \email{jean.jacod@upmc.fr}

 \author{E. Kowalski} \address{ETH Z\"urich -- D-MATH \\ R\"{a}mistrasse
   101\\ 8092 Z\"{u}rich, Switzerland}
 \email{kowalski@math.ethz.ch}

\author{A. Nikeghbali}
 \address{Institut f\"ur Mathematik,
 Universit\"at Z\"urich, Winterthurerstrasse 190,
 CH-8057 Z\"urich,
 Switzerland}
 \email{ashkan.nikeghbali@math.uzh.ch}



\begin{abstract}
  We introduce a new type of convergence in probability theory, which
  we call ``mod-Gaussian convergence''. It is directly inspired by
  theorems and conjectures, in random matrix theory and number theory,
  concerning moments of values of characteristic polynomials or zeta
  functions. We study this type of convergence in detail in the
  framework of infinitely divisible distributions, and exhibit some
  unconditional occurrences in number theory, in particular for
  families of $L$-functions over function fields in the Katz-Sarnak
  framework. A similar phenomenon of ``mod-Poisson convergence'' turns
  out to also appear in the classical Erd\H{o}s-K\'ac Theorem.
\end{abstract}

\maketitle

\section{Introduction} \label{section:Intro} 

Characteristic polynomials of random matrices are essential objects in
Random Matrix Theory, and have also come to play a crucial role in the
remarkable results and conjectures linking random matrices with the
study of $L$-functions in number theory (see e.g. \cite{Mezzadri} for
recent surveys of this connection). Our present work finds its source
in the study of the asymptotic of moments of characteristic
polynomials of random unitary matrices by Keating and
Snaith~\cite{Keating-Snaith}, and the corresponding conjecture for the
moments of the Riemann zeta function on the critical line. More
precisely, Keating and Snaith proved (in probabilistic language) that
if $(Y_N)$, for $N\geq 1$, is a sequence of complex random variables,
with $Y_N$ distributed like $\det(I-X_N)$ for some random variable
$X_N$ taking values in the unitary group $U(N)$ and uniformly
distributed on $U(N)$ (i.e., distributed according to Haar measure),
then for any complex number $\lambda$ with $\Reel(\lambda)>-1$, we
have
\begin{equation}\label{momentinfini}
  \lim_{N\rightarrow\infty}\dfrac{1}{N^{\lambda^{2}}}\E
  \left[|Y_N|^{2\lambda}\right]=
  \dfrac{\left(G\left(1+\lambda\right)\right)^{2}}{G\left(1+2\lambda\right)},
\end{equation}
where $G$ is the \emph{Barnes (double gamma) function}.
\par
Then, using a (now classical) random matrix analogy, they make the
following conjecture for the moments of the Riemann zeta function (see
\cite{Keating-Snaith},\cite{Mezzadri}): for any complex number
$\lambda$ with $\Reel(\lambda)>-1$, we should have
\begin{equation}\label{momentszeta}
  \lim_{T\rightarrow\infty}\dfrac{1}{\left(\log T\right)^{\lambda^{2}}}
  \dfrac{1}{T}\int_{0}^{T}\left|\zeta\left(\dfrac{1}{2}
      +i t\right)\right|^{2\lambda}\dd t
  =M\left(\lambda\right)A\left(\lambda\right)
\end{equation}
where $M\left(\lambda\right)$ is the \emph{random matrix factor},
suggested by~(\ref{momentinfini}), namely
\begin{equation}\label{eq-rmtfactor}
M\left(\lambda\right)
=\dfrac{\left(G\left(1+\lambda\right)\right)^{2}}
{G\left(1+2\lambda\right)}
\end{equation}
while $A\left(\lambda\right)$ is the \emph{arithmetic factor} defined
by the Euler product
\begin{equation}\label{arithmeticfactor}
A\left(\lambda\right)=
\prod_{p}\left(1-\dfrac{1}{p}\right)^{\lambda^{2}}
\left(\sum_{m=0}^{\infty}\left(
\frac{\Gamma(\lambda+m)}{m!\Gamma(\lambda)}\right)^{2}p^{-m}\right),
\end{equation}
where, as usual, $p$ runs over prime numbers, and the product is here
absolutely and locally uniformly convergent; see
Section~\ref{ssec-arith} for details.
\par
We now look at~(\ref{momentinfini}) slightly differently. If we take
$\lambda=\ii u$ to be purely imaginary in (\ref{momentinfini}), then
we obtain a limit theorem involving the \emph{characteristic function}
(or Fourier transform) of the random variables $Z_N=\log|Y_N|^2$ (note
that $|Z_N|\neq0$ almost surely):
\begin{equation}\label{ML}
  \lim_{N\to\infty}e^{u^2\log N}\E[e^{\ii uZ_N}]
=  \lim_{N\to\infty}e^{u^2\log N}\E[e^{\ii u\log|Y_N|^2}]
=\dfrac{\left(G\left(1+\ii u\right)\right)^{2}}{G\left(1+2\ii u\right)}.
\end{equation}
\par
A renormalized convergence of the characteristic function as it occurs
in (\ref{ML}) is not standard in probability theory. However, it has
now appeared in various places in random matrix theory and number
theory, although (to the best of our knowledge) always under the form
of the convergence of normalized Mellin transforms as in
(\ref{momentinfini}) and (\ref{momentszeta}).
\par
In probability theory, the characteristic function $\expect(e^{iuZ})$
is a more natural object to consider, because contrary to Mellin or
Laplace transforms, it always exists and characterizes the
distribution of a random variable $Z$. Hence, in order to look more
deeply into the properties of this type of limiting behavior for a
sequence $(Z_N)$ of real-valued random variables, we use the
characteristic functions (the Fourier transforms of the laws), and
introduce the following definition:

\begin{defn}\label{def:M}
The sequence $(Z_N)$ is said to
\emph{converge in the mod-Gaussian sense} if the convergence
\begin{equation}\label{M}
e^{-\ii u\beta_N+u^2\gamma_N/2}\E[e^{\ii u Z_N}]~\to~\Phi(u)
\end{equation}
holds for all $u\in\RR$, where $\beta_N\in\RR$ and $\gamma_N\geq0$
are two sequences and $\Phi$ is a complex-valued function which is
continuous at $0$ (note that necessarily $\Phi(0)=1$). We call
$(\beta_N,\gamma_N)$ the {\em parameters}, and $\Phi$ the associated
{\em limiting function}.
\end{defn}

The main aim of this paper is to provide a general framework in which
convergence such as (\ref{M}) occurs naturally. Secondary goals are to
give examples, both in probability and number theory, and to argue for
the interest of this notion. As a first example, of course,~(\ref{ML})
shows that the random variables $\log |Y_N|^2$ converge in
mod-Gaussian sense with parameters $(0,2\log N)$ and limiting function
$G(1+iu)^2G(1+2iu)^{-1}$.
\par
\medskip
\par
The paper is organized as follows. In
Section~\ref{sec-mod-Gaussian}, we properly define the
``mod-Gaussian convergence'', give some immediate properties and
describe some easy examples where it occurs. In
Section~\ref{subsect:a} we show that, under some conditions on
the third moment, mod-Gaussian convergence occurs for sums of C\'esaro
means of triangular arrays of independent random variables. Within this
framework, a characterization of the limiting function $\Phi$
is found in the case of infinitely divisible
distributions, and it is shown to have a representation of
L\'evy-Khintchine type, with one extra term. More generally, we also
give a necessary and sufficient condition for the mod-Gaussian
convergence to hold when the laws of the $(Z_N)$ are infinitely
divisible, with an explicit expression for the limit function $\Phi$
(which again has a L\'evy-Khintchine type representation, with two
extra terms now).

In Section~\ref{sec-nt}, we give examples of mod-Gaussian convergence
in number theory, in two directions. First, we show that the
arithmetic factor $A(\lambda)$ in the moment conjecture, for
$\lambda=iu$, arises as limiting function $\Phi(u)$ for the
mod-Gaussian convergence of very natural sequences of random
variables, and hence so does $M(iu)A(iu)$; this is in particular
additional (though modest) evidence in favor of the
conjecture~(\ref{momentszeta}), since if that were not the case, the
conjecture would necessarily be false.  Second, we explain how
Deligne's Equidistribution Theorem and the Katz-Sarnak philosophy lead
to the proof of an analogue of the moment conjecture for families of
$L$-functions over function fields; this second problem was raised in
particular by B. Conrey.
We think that these facts illustrate that the philosophy of
mod-Gaussian convergence is a potentially crucial analytic framework
underlying deep issues of number theory. In addition, we interpret the
classical Erd\H{o}s-K\'ac theorem in a similar way, although with
``mod-Poisson'' convergence.
\par
\begin{rem}
  Because of the possible relevance of this paper both to probability
  theory and number theory, we have tried to write it in a balanced
  manner, so that experts in either field can understand it. This
  means, in particular, that we recall precisely some facts which, for
  one field at least, are entirely standard and well-known (e.g.,
  facts about infinitely divisible distributions, or zeta functions of
  curves over finite fields). This also means that, even though we are
  aware of the possibility of extending our results to sharper
  statements, we have not done so when this would, in our opinion,
  obscure the main ideas for one half of our readers.
\end{rem}
\par
\textbf{Notation.} We use the notation
$$
\nu(f),\quad \quad \int{f(x)d\nu(x)},\quad\quad
\int{f(x)\nu(\dd x)}
$$
interchangeably for the integral of a function $f$ with respect to
some measure $\nu$. We write, as usual in probability, $x\wedge y$ for
$\min (x,y)$. In number-theoretic contexts, $p$ always refers to a
prime number, and sums and products over $p$ (with extra conditions)
are over primes satisfying those conditions.
\par
\medskip
\par
\textbf{Acknowledgments}. We thank M. Yor for a number of interesting
discussions related to this project. 
\par
A.N. was partially supported by SNF Schweizerischer Nationalfonds
Projekte Nr. 200021 119970/1.

\section{General properties of mod-Gaussian
  convergence}\label{sec-mod-Gaussian}

We start with the definition of a slightly stronger form of
mod-Gaussian convergence:

\begin{defn}\label{def:MMM}
The sequence $(Z_N)$ is said to
\emph{strongly converge in the mod-Gaussian sense} if the convergence
in (\ref{M}) holds uniformly in $u$, on every compact subset of $\RR$.
\end{defn}

The left side of (\ref{M}) being continuous, the strong convergence
implies that $\Phi$ is continuous, hence the mere convergence.

For the mod-Gaussian convergence with parameters $\beta_N=\gamma_N=0$,
the convergence and the strong convergence are the same, and amount to
the convergence in law of the variables $Z_N$ (this is basically
L\'evy's Theorem).

\subsection{Formal properties}

It is natural to first ask for the intuitive meaning of mod-Gaussian
convergence (and for explanation of the chosen terminology). The
following proposition describes what might be called ``regular''
mod-Gaussian convergence:

\begin{prop}\label{prop:firstprop}
  Let $(X_N)$ be a sequence of real random variables converging in law
  to a limiting variable with characteristic function $\Phi$. If for
  each $N$ we let
\begin{equation}\label{eq-regular}
Z_N=X_N+G_N,
\end{equation}
where $G_N$ is a Gaussian random variable independent of $X_N$, and
with mean $\beta_N$ and variance $\gamma_N$, then we have the strong
mod-Gaussian convergence of the sequence $(Z_N)$, with limiting
function $\Phi$ and parameters $(\beta_N,\gamma_N)$.
\end{prop}

\begin{proof}
  Because $X_N$ and $G_N$ are independent, we have
$$
\E(e^{iu Z_N})=\E(e^{iu X_N})\E(e^{iu G_N})=
e^{iu\beta_N-u^2 \gamma_N/2}\E(e^{iu X_N}),
$$
by the formula for the characteristic function of a Gaussian random
variable (see, e.g.,~\cite[(16.2)]{jacodprotter}). 
\par
The convergence in law of $(X_N)$ implies the local uniform
convergence of the characteristic function of $X_N$ to $\Phi$
(the easy half of the L\'evy Criterion), hence the convergence
(\ref{M}) holds locally uniformly in $u$.
\end{proof}

We see that under this scheme, the variable $Z_N$ is decomposed into
two terms: a variable $X_N$, which converges in law, and a Gaussian
variable, with arbitrary variance and mean, which can be viewed as a
``noise'' added to the converging variables. We then think intuitively
of looking at $Z_N$ \emph{modulo} the subset of Gaussian random
variables, and then only the convergent sequence remains. It is this
way of producing the convergence introduced in Definition~\ref{def:M}
which motivated the terminology ``mod-Gaussian''.
\par
Proposition
\ref{prop:firstprop} does \emph{not} cover all cases of mod-Gaussian
convergence: as we will see later, the limiting function $\Phi$ of
a sequence converging in the mod-Gaussian sense may not be a
characteristic function. However, the intuitive picture of some
converging ``core'' hidden by possibly wilder and wilder noise may
still be useful.
\par
The next proposition summarizes a few basic properties of
mod-Gaussian convergence that follow easily from the definition
(the first part, in particular, is another justification for the
terminology).

\begin{prop}\label{prop-formal} 
  \emph{(1)} Let $(Z_N)$ be a sequence of real-valued random variables
  for which mod-Gaussian convergence holds with parameters
  $(\beta_N,\gamma_N)$ and limiting function $\Phi$. Then the
  mod-Gaussian convergence holds for some other parameters
  $(\beta'_N,\gamma'_N)$ and limiting function\footnote{\ Here,
    $\Phi'$ is not the derivative of $\Phi$.} $\Phi'$, if and only if
  the limits
\begin{equation}\label{A1}
\beta=\lim_{N\ra+\infty}{(\beta_N-\beta'_N)},\quad\quad
\gamma=\lim_{N\ra +\infty}{(\gamma_N-\gamma'_N)},
\end{equation}
exist in $\RR$. In this case $\Phi'$ is given by
\begin{equation}\label{A2}
\Phi'(u)=e^{i\beta u-u^2\gamma /2}\,\Phi(u),
\end{equation}
and if the strong convergence holds with the parameters
$(\beta_N,\gamma_N)$ it also holds with $(\beta'_N,\gamma'_N)$.
\par
\emph{(2)} Let $(Z_N)$ and $(Z'_N)$ be two sequences of random
variable with mod-Gaussian convergence (resp. strong convergence),
with respective parameters $(\beta_N,\gamma_N)$ and
$(\beta'_N,\gamma'_N)$, and limiting functions $\Phi$ and $\Phi'$. If
$Z_N$ and $Z'_N$ are independent for all $N$, then the sums
$(Z_N+Z'_N)$ satisfy mod-Gaussian convergence (resp. strong
convergence) with limiting function the product $\Phi\Phi'$ and
parameters $(\beta_N+\beta'_N,\gamma_N+\gamma'_N)$.
\end{prop}

\begin{proof}
(2) follows from the multiplicativity of the characteristic functions
of independent variables.
\par
As for (1), if (\ref{A1}) holds the mod-Gaussian convergence with
parameters $(\beta'_N,\gamma'_N)$ and limiting function $\Phi'$ given
by (\ref{A2}) is obvious from the definition, as well as the last
claim (it is also a special case of (2)).
\par
Conversely, suppose that the mod-Gaussian convergence holds with
parameters $(\beta'_N,\gamma'_N)$ and limiting function $\Phi_1$. Then
\begin{equation}\label{A3}
  e^{-iu\beta_N+u^2\gamma_N/2}\E(e^{iuZ_N})\ra \Phi(u),\quad
  e^{-iu\beta'_N+u^2\gamma'_N/2}\E(e^{iuZ_N})\ra \Phi'(u).
\end{equation}
The function $\Phi$ and $\Phi'$ are both continuous and equal to $1$
at $0$, so they are non-vanishing on a neighborhood $[-\delta,\delta]$
of $0$, with $\delta>0$. Then for any non-zero $u\in[-\delta,\delta]$,
by taking the ratio and the modulus in (\ref{A3}), we get
$$
\lim_{N\ra+\infty}{e^{(\gamma'_N-\gamma_N)u^2/2}}
=\Bigl|\frac{\Phi'(u)}{\Phi(u)}\Bigr|>0,
$$
hence the second part of (\ref{A1}) holds with
$$\gamma=\frac{2}{u^2}
\log \Bigl|\frac{\Phi'(u)}{\Phi(u)}\Bigr|
$$
(this limit does not depend on the choice of $u$). Moreover, by
(\ref{A3}) again,
$$
e^{-i(\beta_N-\beta'_N)u}\ra \frac{\Phi(u)}{\Phi'(u)}~e^{u^2\gamma/2}
$$
for all $u\in[-\delta,\delta]$, and the first part of (\ref{A1}) follows.
\end{proof}

\begin{rem} 
It is important to notice that the mod-Gaussian
convergence does {\em not}\, require the parameter sequences
$\beta_N$ and $\gamma_N$ to converge. However, it implies the
uniqueness of the parameters $(\beta_N,\gamma_N)$, {\em up to a
convergent sequence} (this is what (1) above says).
\par
In the most usual situations, $\Phi$ will be smooth and
$\expect(e^{iuZ_N})$ also, and comparing expansions to second order
for the left-hand side and right-hand side of~(\ref{M}) around $u=0$,
one finds that, in this situation, the following will hold:
\par
(1) When varying the parameters (using~(\ref{A2}), there is a
\emph{unique} possible limiting function $\Phi_0$ such that
\begin{equation}\label{eq-canonical}
\Phi_0(u)=1+o(u^2),\quad\quad \text{ for } u\ra 0\;
\end{equation}
\par
(2) For this limiting function $\Phi_0$, up to adding sequences
converging to $0$, we have
$$
\beta_N=\expect(Z_N),\quad\quad
\gamma_N=\variance(Z_N).
$$
\par
However, note that in natural situations, it is by no means clear if
the limiting function satisfies~(\ref{eq-canonical}), for instance
for~(\ref{ML}). It may also not be the most natural choice (see
Proposition~\ref{prop:firstprop}).
\end{rem}

\begin{rem}
  Observe that (1) in Proposition~\ref{prop-formal} would fail, should
  we drop the requirement of continuity of $\Phi$ at $0$. For example
  if $(Z_N)$ converges in the mod-Gaussian sense with parameters
  $(\beta_N,\gamma_N)$, with $\gamma_N\to\infty$, and if we take
  $\delta_N\to\infty$ with $0\leq \delta_N<\gamma_N$, then (\ref{M})
  holds with the parameters $(\beta_N,\gamma_N-\delta_N)$ as well, and
  the associated limiting function vanishes outside $0$.
\end{rem}

\subsection{Remarks, questions and problems} \label{sec-questions}

The introduction of mod-Gaus\-sian convergence suggests quite a few
questions which it would be interesting to answer, to deepen the
understanding of the meaning of this type of limit behavior of
sequences of random variables.
\par
We first remark that, from the point of view of
Proposition~\ref{prop:firstprop}, it is also natural to introduce
convergence modulo other particular classes of random variables: given
a family $\mathcal{F}=(\mu_{\lambda})_{\lambda\in\Lambda}$ of
probability distributions parametrized by some set $\Lambda$, such
that the Fourier transforms
$$
\hat{\mu}(\lambda,u)=\int_{\Rr}{e^{itx}d\mu_{\lambda}(t)}
$$
are non-zero for all $u\in\Rr$, one would say that a sequence of
random variables $(Z_N)$ converges in the mod-$\mathcal{F}$ sense if,
for some sequence $\lambda_N\in \Lambda$, we have
$$
\lim_{N\ra +\infty}{
  \hat{\mu}(\lambda_N,u)^{-1}\E(e^{iuZ_N})}=\Phi(u)
$$
for all $u\in\Rr$, the limiting function $\Phi$ being continuous at
$0$. Weak and strong convergence can be defined accordingly.
A particularly natural idea that comes to mind is to look at the
family of {\em symmetric stable}\, variables with index $\alpha\in(0,2)$,
and parametrized by $\gamma\in \Lambda=(0,+\infty)$, so that
$$
\hat{\mu}(\gamma,u)=e^{-\gamma|u|^{\alpha}}.
$$
\par
Such examples for $\alpha\not=2$ have not (yet) been observed ``in the
wild'', but we will see in Section~\ref{ssec-poisson} that classical
results of analytic number theory can be interpreted as an instance of
``mod-Poisson'' convergence, i.e., with $\mathcal{F}$ the family of
Poisson distributions on the integers with parameter
$\lambda\in(0,+\infty)$.
\par
Another natural generalization is the fairly obvious notion of
multi-dimen\-sional mod-Gaussian convergence for random vectors, or
indeed for stochastic processes. The finite-dimensional case may be
used, in random-matrix context, to interpret results on moments of
products of characteristic polynomials evaluated at different points
of the unit circle.  
\par
\medskip
\par
Now here are some obvious questions:
\par
(1) Can one find a convenient criterion for mod-Gaussian convergence
to be of the ``regular'' type of Proposition~\ref{prop:firstprop}? Is
there a ``weak convergence'' description of mod-Gaussian convergence
using test functions of some type?
\par
(2) Is the idea useful in analysis also? Here one could think that the
analogue of the ``regular'' mod-Gaussian convergence would be to have a
sequence of distributions $(T_N)$ satisfying
$$
T_N=g_N\star S_N
$$
where $\star$ is the convolution product, $g_N$ is the distribution
associated to a Schwartz function of the type
$\exp(iu\beta_N-u^2\gamma_N/2)$, and $S_N$ is some convergent sequence
of distributions.  For instance, is it possible that some
approximation schemes for solutions of some kind of equations converge
in the mod-Gaussian sense? Is there, then, a more direct way to
recover $S_N$ (numerically, for instance) than by performing an
inverse Fourier transform? In other words, can the Gaussian noise
$g_N$ be ``filtered out'' naturally?
\par
(3) Is there a convenient criterion for a function $\Phi$ defined on
$\Rr$, with $\Phi(0)=1$ and $\Phi$ continuous at $0$, to be a limiting
function for mod-Gaussian convergence, similar to Bochner's Theorem (a
function $\varphi\,:\, \Rr\ra \Cc$ is a characteristic function of a
probability measure if and only if $\varphi$ is continuous,
$\varphi(0)=1$, and $\varphi$ is a positive-definite function)?


\section{Limit theorems with mod-Gaussian behavior}
\label{subsect:a}

In this section, we provide several theorems characterizing situations
in which mod-Gaussian convergence, as introduced in Definition
\ref{def:M}, holds.

\subsection{The central limit theorem for mod-Gaussian convergence}
\label{sec-limit}

We start with a result which provides a very general criterion for
mod-Gaussian convergence in the framework of the classical limit
theorems of probability theory. This suggests that mod-Gaussian
convergence is a ``higher order'' analogue of the classical
convergence in distribution. 
\par
We recall first the standard limit theorems in the setting of
triangular arrays of independent identically distributed (in short,
i.i.d.) random variables.
\par
Let $(X_i^n)$, for $n\geq 1$ and $1\leq i\leq n$, be random
variables, where the variables
$$
X_1^n,\ldots,X_n^n
$$
in each row are i.i.d. with law denoted by $\mu_n$. For any
integer $n\geq 1$, let
$$
S_n=X_1^n+\ldots+X_n^n
$$
denote the sum of the $n$-th row.
\par
If the $X_i^n$ have expectation zero and variance $1$, the
Law of Large Numbers states that $S_n/n$ converges in probability to
$0$ as $n\ra +\infty$, and the Central Limit Theorem
states that the random variables $S_n/\sqrt{n}$, which are centered
with variance $1$, converge in law to the standard Gaussian variable
$\mathcal{N}(0,1)$.
\par
The latter can be interpreted as a ``second order'' type of behavior
of $S_n/n$, beyond the ``first order'' convergence to $0$, by
rescaling by $\sqrt{n}$ to obtain variance $1$. Instead of this
classical normalization, we want to look for finer information by
normalizing with a growing variance. 
\par
Working with $S_n$ directly does not seem to lead to fruitful results,
but for $N\geq 1$, we can consider the \emph{logarithmic mean} of the
$S_n$, i.e., the random variables
\begin{equation}\label{eq-cesaro}
Z_N=\sum_{n=1}^N{\frac{S_n}{n}}=
\sum_{n=1}^{N}\dfrac{1}{n}\left(X_1^n+\ldots+X_n^n\right)
\end{equation}
which have variance given by the $N$-th harmonic number $H_N$, i.e.,
we have
$$
\V(Z_N)=H_N=\sum_{n=1}^N{\frac{1}{n}}.
$$
\par
Note that it is well-known that, for a numerical sequence $(u_n)$ that
converges to a limit $\alpha$, the analogue \emph{logarithmic means}
$$
v_N=\frac{1}{\log N}\sum_{n=1}^N{\frac{u_n}{n}}
$$
also converge to $\alpha$ (see, e.g.,~\cite[III.9]{zygmund}). This
shows that, intuitively, the $Z_N$ can ``amplify'' the sums $S_N$ by a
logarithmic factor.
\par
We now show that, under quite general conditions, the $Z_N$ converge
in the mod-Gaussian sense.


\begin{thm}\label{thm:1}
  Let $\left(X_i^n\right)_{i.n\geq1}$ be a triangular array of random
  variables, \emph{all independent}, and such that the variables in
  the $n$th row have the same law $\mu_n$, and let assume that $\mu_n$
  has mean zero, variance $1$ and third absolute moment satisfying
\begin{equation}\label{conditionmomet3}
  \sum_{n=1}^{\infty}\dfrac{m_n}{n^2}<\infty, 
  \quad  \text{ where } \quad m_n=\int_{\Rr}|x|^3\mu_n(\dd x).
\end{equation}
Then the logarithmic means $Z_N$ defined by~\emph{(\ref{eq-cesaro})}
strongly converge in the mod-Gaussian sense, with parameters
$(0,H_N)$, or with parameters $(0,\log N)$.
\end{thm}

For the proof, we recall the following useful lemma.

\begin{lem}[\cite{breiman}, Proposition 8.44, p.180]\label{breiman1}
  If $X$ is a random variable with $\E[|X|^k]<\infty$, then the
  characteristic function $\phi$ of $X$ has the expansion:
$$
\phi(u)=\sum_{j=0}^{k-1}\dfrac{(\ii u)^j}{j!}\E[X^j]+
\dfrac{(\ii u)^k}{k!}\left(\E(X^k)+\delta(u)\right),
$$
where $\delta(u)$ is a function of $u$ that satisfies
$\lim_{u\to0}\delta(u)=0$ and $|\delta(u)|\leq3\E(|X|^k)$ for all $u$.
\end{lem}

\begin{proof}[Proof of Theorem~\ref{thm:1}]
  Let $\phi_n(u)$ be the characteristic function of $\mu_n$, and let
$$
G_N(u)=e^{u^2 H_N/2}\E[e^{\ii u Z_N}],
$$
which is continuous in $u$. It then suffices to show that $G_N$
converges locally uniformly to a limiting function.  By the
independence assumption and standard properties of characteristic
functions, we have
$$
G_N(u)=e^{u^2H_N/2}\prod_{n=1}^N \phi_n(u/n)^n.
$$
\par
Let $A>0$ be fixed. Applying Lemma \ref{breiman1} with $k=2$ and
using the fact that the $\mu_n$'s are centered with variance $1$, we
obtain that for $|u|\leq A$ and $n\geq 2A$, we have
$$
|\phi_n(u/n)-1|\leq 2u^2/n^2\leq 1/2.
$$
\par
Taking $\log$ to be the principal branch
of the logarithm (which is zero at $1$) on the disk
$\{z\in\CC:\;|z-1|\leq1/2\}$, we have for $N>M\geq2A$ and $|u|\leq A$:
\begin{equation}\label{eqcvuni}
  G_N(u)=G_M(u)e^{H_{M,N}(u)}
\end{equation}
where
\begin{equation}\label{defHMN}
H_{M,N}(u)=\sum_{n=M+1}^Nn\left(\log \phi_n(u/n)+\dfrac{u^2}{2n^2}\right).
\end{equation}
\par
Another application of Lemma \ref{breiman1} with $k=3$ combined with
the inequality $|\log(1+z)-z|\leq 4|z|^2$ for $|z|\leq1/2$, yields
for $n\geq M$:
$$
\Bigl|\log\phi_n(u/n)+\dfrac{u^2}{2n^2}\Bigr|
\leq\dfrac{u^3}{n^3}m_n +\dfrac{16u^4}{n^4}.
$$
\par
It now follows from the assumption~(\ref{conditionmomet3}) that for
any fixed $M\geq2A$, $H_{M,N}(u)$ is the partial sum (in $N$) of a
series starting at $M+1$, which converges uniformly in $u\in[-A,A]$ to
a limit $\tilde{H}_M(y)$, as $N\to\infty$. Consequently, we deduce
from (\ref{eqcvuni}) that $G_N(u)$ converges, as $N\to\infty$, to
$G_M(u)\,\exp \tilde{H}_N(u)$, uniformly in $u\in[-A,A]$. Since $A$ is
arbitrarily large, the result follows.
\end{proof}

As an easy consequence of Theorem \ref{thm:1} one obtains the
following central limit theorem for the sum of C\'esaro means as
introduced in Theorem \ref{thm:1}.

\begin{cor}\label{cor-clt}
  With the assumptions and notations of Theorem \ref{thm:1}, the
  re-scaled random C\'esaro means $Z_N/\sqrt{\log N}$ converge in law
  to the standard Gaussian law $\mathcal{N}(0,1)$.
\end{cor}

\begin{proof}
With the notation $G_N$ of the previous proof, the characteristic
function of $Z_N/\sqrt{\log N}$ is
$$e^{-u^2H_N/\log N}~G_N(u/\sqrt{\log N}).$$
Now $G_N$ converges locally uniformly to a function equal to $1$
at $0$, hence $G_N(u/\sqrt{\log N})\to1$, whereas we have
$H_N\sim\log N$. So the result follows from L\'evy's theorem.
\end{proof}


\begin{example}\label{ex-1}
  The observation of the following example, arising from Random Matrix
  Theory, was the source and motivation for the considerations in this
  and the next sections.
\par
Consider a sequence $(\gamma_n)$ of independent random variables,
with $\gamma_n$ having a gamma distribution with scale parameter $1$
and index $n$, that is with the density $\frac{1}{\Gamma(n)}x^{n-1}e^{-x}
\,\charfun_{\RR_+}(x)$. We are interested in the behavior of
$$Z_N=\sum_{n=1}^N{\frac{\gamma_n}{n}}-N.
$$
Recalling that one can represent $\gamma_n$ as the sum
$\gamma_n=\sum_{i=1}^nY^n_i$, where the $Y^n_i$ for $i=1,\cdots,n$ are
i.i.d.\ with gamma distribution with index $1$ (or, ``exponential
distribution''), we see that $Z_N$ is associated by (\ref{eq-cesaro})
with the variables $X^n_i=Y^n_i-1$, which have mean $0$ and variance
$1$ and a finite third moment (so (\ref{conditionmomet3})
holds). Hence we obtain the limit formula:
\begin{equation}\label{exemplegamma}
  \lim_{N\to\infty}e^{u^2H_N/2}
  \E\left[\exp\left(\ii u\left(
        \sum_{n=1}^{N}\dfrac{\gamma_{n}}{n}-N\right)
    \right)\right]=\Phi(u)
\end{equation}
for some continuous function $\Phi(u)$ with $\Phi(0)=1$.
\par
Now it turns out that this limit can also be computed explicitly (as
was first done in~\cite[Th. 1.2]{NikYor}), as a consequence of a
decomposition of the probability law of the characteristic polynomial
of random unitary matrices (distributed according to Haar measure) as
product of independent gamma and beta random variables. Indeed, it was
shown that for any complex number $z$ with $\Reel(z)>-1$, we have
\begin{multline}\label{premiermellinlimit}
  \lim_{N\rightarrow\infty}\dfrac{1}{N^{\frac{z^{2}}{2}}}\E
  \left[\exp\left(-z\left(\sum_{n=1}^{N}\dfrac{\gamma_{n}}{n}
        -N\right)\right)\right]=\\
  \left(A^{z}\exp\left(\dfrac{z^{2}}{2}\right)G\left(1+z\right)\right)^{-1}
\end{multline}
where $ A=\sqrt{\dfrac{e}{2\pi}}$ and $G$ is the Barnes function.
\par
When $z=-\ii u$, the left-hand side is just the left-hand side
of~(\ref{exemplegamma}), so this argument gives a formula for
$\Phi(u)$. More precisely, a representation of the Barnes function
obtained in~\cite[Proposition 2.3]{NikYor}, shows that the limiting
function $\Phi(u)$ in~(\ref{exemplegamma}) is equal to
\begin{equation}\label{LKrep}
\exp\left(-\frac{u^2}{2}+
    \int_0^\infty\dfrac{1}{x(2\sinh(\tfrac{x}{2}))^2}\left(e^{\ii u x}
      -1-\ii u x+\dfrac{u^2x^2}{2}\right)dx\right).
\end{equation}
\par
This formula, for a probabilist, is very strongly reminiscent of the
\emph{L\'evy-Khintchine representation} for the characteristic
function of infinitely divisible distributions. Since the gamma
variables are themselves infinitely divisible, it is natural to wonder
whether this example has natural generalizations to other such random
variables. We will now see that this is indeed the case.
\end{example}

\subsection{Infinitely divisible distributions}\label{ss-ID}

In this section, we will refine Theorem~\ref{thm:1}, by finding an
expression for the limiting function $\Phi(u)$. The context in which
we do this is that of \emph{infinitely divisible distributions}. Since
readers with number-theoretic background, in particular, may not be
familiar with this theory, we first recall some basic facts, referring
to the standard textbooks~\cite{breiman} and~\cite{Sato} for proofs
and further details.

\begin{defn}[\cite{Sato}, p. 31]
  Denote by $\mu^{n*}$ the $n$-fold convolution of a probability
  measure $\mu$ with itself. The probability measure $\mu$ on
  $\mathbb{R}$ is said to be infinitely divisible if for any positive
  integer $n$, there is a probability measure $\mu_n$ on $\mathbb{R}$
  such that $\mu=\mu_n^{n*}$.
\end{defn}

\begin{thm}[\cite{Sato}, Theorem 8.1, p. 37]
The following properties hold:
\par
\emph{(1)} If $\mu$ is an infinitely divisible distribution on
  $\Rr$, then for $u\in\Rr$, we have
\begin{multline}\label{levykintchine}
  \widehat{\mu}(u)\equiv \int_{-\infty}^\infty e^{\ii ux}\mu(\dd x)
  =\exp\left[-\frac{1}{2}\sigma u^2+\ii\beta u\right.
  \\
  \left.+\int_\mathbb{R}\left(e^{\ii ux}-1-\ii
      ux\I_{|x|\leq1}\right)\nu(\dd x)\right]
\end{multline}
where $\sigma\geq0$, $\beta\in\RR$ and $\nu$ is a measure on $\RR$,
called the L\'evy measure, satisfying
\begin{equation}\label{mesurelevy}
  \nu(\{0\})=0 \text{ and } \int_\RR \left(x^2\wedge1\right)\nu(\dd x)<\infty.
\end{equation}
\par
\emph{(2)} The representation of $\widehat{\mu}(u)$
in~\emph{(\ref{levykintchine})} by $\sigma$, $\beta$ and $\nu$ is
unique.
\par
\emph{(3)} Conversely, if $\sigma\geq0$, $\beta\in\RR$ and $\nu$ is a
measure satisfying~\emph{(\ref{mesurelevy})}, then there exists an
infinitely divisible distribution $\mu$ whose characteristic function
is given by~\emph{(\ref{levykintchine})}.
\par
The parameters $(\sigma,\beta, \nu)$ are called the generating triplet
of $\mu$.
\end{thm}

\begin{rem}
  If we compare~(\ref{LKrep}) with~(\ref{levykintchine}), we see that
  the formula~(\ref{LKrep}) arising from the mod-Gaussian convergence
  in Example~\ref{ex-1} is \emph{not} an actual L\'evy-Khintchine
  formula, because the function $x^2\wedge 1$ is not integrable with
  respect to the measure $\frac{\dd x}{x(2\sinh(x/2))^2}$ (although
  $|x|^3\wedge1$ is); this explains why an additional second order
  term is required in the integrand.
\end{rem}

\begin{rem}
  The theorem above is the usual representation of the Fourier
  transform of an infinitely divisible probability distribution. There
  are many other ways of getting an integrable integrand with respect
  to the L\'evy measure $\nu$, and this will be important for us (see
  the discussion in~\cite[p. 38]{Sato}). 
\par
Let $h$ be a {\em truncation function}, that is a real function on
$\RR$, bounded, with conpact support, and such that $h(x)=x$ on a
neighborhood of $0$. Then for every $u\in\RR$, $x\mapsto(e^{\ii
  ux}-1-uh(x))$ is integrable with respect to $\nu$, and
(\ref{levykintchine}) may be rewritten as:
\begin{equation}\label{levykintchinetronc}
  \widehat{\mu}(u)=
\exp\left[-\frac{1}{2}\sigma u^2+\ii\beta_h u
+\int_\mathbb{R}\left(e^{\ii ux}-1-\ii uh(x)\right)\nu(\dd x)\right]
\end{equation}
where
$$
\beta_h=\beta+\int_{-\infty}^\infty(h(x)-x\I_{|x|\leq1})\nu(\dd x).
$$ 
\par
The triplet $(\sigma, \beta_h,\nu)$ is called the generating triplet
of $\mu$ with respect to the truncation function $h$.
\end{rem}

One can also express the moments of $\mu$ in terms of the L\'evy
measure $\nu$.  This is dealt with in Section 25 (p. 159 onwards)
in~\cite{Sato}. For example one can show that $\int_\RR|x|\mu(\dd
x)<\infty$ if and only if $\int_{|x|>1}|x|\nu(dx)<\infty$. More
generally, the measure $\mu$ admits a moment of order $n$ if and only
if $\int_{|x|>1}|x|^n\nu(\dd x)<\infty$; in this case, the cumulants
$(c_k)$ are related to the moments of $\nu$ as follows:
\begin{align*}
  c_1&\equiv\int_\RR x\mu(\dd x) = \gamma+\int_{|x|>1}x\nu(\dd x), \\
  c_2&\equiv\int_\RR x^2\mu(\dd x)-c_1^2 = 
\sigma+\int_{-\infty}^\infty x^2\nu(\dd x), \\
  c_k &= \int_{-\infty}^\infty x^k\nu(\dd x),\quad
\text{ for } 
\quad
3\leq k\leq n.
\end{align*}
\par
In particular, if $\mu$ admits a first moment, then
$\int_{|x|>1}|x|\nu(\dd x)<\infty$ and we can write:
\begin{equation}
\widehat{\mu}(u)
=\exp\left[-\frac{1}{2}\sigma u^2+\ii c_1 u
+\int_\mathbb{R}\left(e^{\ii ux}-1-\ii ux\right)\nu(\dd x)\right].
\end{equation}
\par

\subsection{Mod-Gaussian convergence in the case of infinitely
  divisible variables}

Motivated by Example~\ref{ex-1}, we now consider the setting of
Section~\ref{sec-limit} when the random variables forming the
triangular array $(X_i^n)$ are infinitely divisible. All other
assumptions (concerning expectation, variance, third moment) remain in
force; recall that $\mu_n$ is the law of the $n$-th row of the
array and $\phi_n$ its characteristic function.
\par
We will see that, in fact, when the probability measures $\mu_n$ are
infinitely divisible, we can give an explicit representation of the
limiting function $\Phi$ of Theorem~\ref{thm:1} in terms of the
generating triplets for the measures $\mu_n$. 
\par
Indeed, using the results recalled in subsection \ref{ss-ID}, it is
easily seen that in the current situation, we can write
\begin{equation}\label{mun:1}
\phi_n(u)=\exp(\psi_n(u)),
\end{equation}
where
\begin{equation}\label{mun:2}
\psi_n(u)=-\dfrac{\sigma_n u^2}{2}+
\int_{-\infty}^{\infty}\left(e^{\ii u x}-1-\ii u x\right)\nu_n(\dd x),
\end{equation}
and
\begin{gather}
  \int_{-\infty}^{\infty}x^2\nu_n(\dd x)=1-\sigma_n\in [0,1], 
\quad\quad
  \int_{-\infty}^{\infty}|x|^3\nu_n(\dd x)<\infty, \\
  \sum_{n=1}^{\infty}\dfrac{1}{n^2}\int_{-\infty}^{\infty}|x|^3\nu_n(\dd
    x)<\infty. \label{mun:5}
\end{gather}

We now find a generalization of~(\ref{LKrep}) in the present situation.

\begin{thm}\label{thm:2} In the setting of Theorem \ref{thm:1},
assume further that the probability measures $\mu_n$ are infinitely
  divisible and satisfy the conditions \emph{(\ref{mun:1})}
  up to~\emph{(\ref{mun:5})}. Then the sequence $(Z_N)$ strongly converges
in the mod-Gaussian sense, with the parameters $(0,H_N)$ and
limiting function $\Phi=e^\Psi$, where
$$
\Psi(u)=\int_{-\infty}^{\infty}\left(e^{\ii u x}-1-\ii u x
+\dfrac{u^2 x^2}{2}\right)\nu(\dd x)
$$
and\begin{equation}
\nu=\sum_{n=1}^{\infty}n\nu'_n,
\end{equation}
and the measures $\nu'_n$ are defined by
$$
\nu'_n(A)=\int_{-\infty}^{\infty}\I_{A}(x/n)\nu_n(\dd x),
$$
for any Borel set $A$. Consequently, $\nu$ is a positive measure which
satisfies $\nu(\{0\})=0$ and $\int_{-\infty}^{\infty}|x|^3\nu(\dd
x)<\infty$.
\end{thm}

\begin{proof}
  With the notations of the proof of Theorem \ref{thm:1}, and
  combining (\ref{eqcvuni}), (\ref{defHMN}) and
  (\ref{mun:1})-(\ref{mun:5}) we have:
\begin{equation}
H_{0,N}(u)=\sum_{n=1}^N n\int_{-\infty}^{\infty}\left(e^{\ii u x}-1
-\ii u x+\dfrac{u^2 x^2}{2}\right)\nu'_n(\dd x)
\end{equation}
and the result of the theorem follows immediately, with
$$
\int_{-\infty}^{\infty}|x|^3\nu(\dd x)
=\sum_{n=1}^{\infty}\frac{1}{n^2}
\int_{-\infty}^{\infty}|x|^3\nu_n(\dd x).
$$
\end{proof}


\subsection{A criterion for mod-Gaussian convergence}

We now want to prove a general result characterizing the existence of
mod-Gaussian convergence of a sequence $(Z_N)$ such that the random
variables $Z_N$ are infinitely divisible.  This is the analogue
of the classical criterion for convergence in distribution of infinitely
divisible random variables due to Gnedenko and Kolmogorov
(see for example ~\cite[Th. 8.7]{Sato}).

\begin{thm}\label{thm:3}
  Let $(Z_N)$ be a sequence of real-valued random variables whose
  respective laws $\mu_N$ are infinitely divisible, with generating
  triplets $(\sigma_N,b_N,\nu_N)$ relative to a fixed continuous
  truncation function $h$.
\par
Then we have mod-Gaussian strong convergence \emph{if and only
if}\, the following two conditions hold:
\par
\emph{(1)} The sequence
$\kappa_N=\int_{-\infty}^{\infty}h(x)^3\nu_N(\dd x)$
converges to a finite limit $\kappa$;
\par
\emph{(2)} There exists a nonnegative measure $\nu$ satisfying
$$
\nu(\{0\})=0,\quad
\int_{\Rr}{(x^4\wedge1 )\nu(\dd x)}<+\infty
$$
and such that $\nu_N(f)\to\nu(f)$ for any continuous function $f$ with
$|f(x)|\leq C(x^4\wedge1)$ for $x\in\RR$ and some constant $C\geq 0$.
\par
Under these conditions, one may take the parameters
\begin{equation}\label{eq-select}
\beta_N=b_N,\quad\quad
\gamma_N= \sigma_N+\nu_N(h^2),
\end{equation}
and the limiting function is then $\Phi=\exp(\Psi)$, where
\begin{equation}
  \Psi(u)=-\ii\dfrac{u^3}{6}\kappa+
  \int_{-\infty}^{\infty}\left(e^{\ii u x}-1-\ii u h(x)+\dfrac{u^2 h(x)^2}{2}+
    \ii\dfrac{u^3 h(x)^3}{6}\right)\nu(\dd x). \label{rep:2}
\end{equation}
\end{thm}

\begin{proof}
  The left-hand side of (\ref{M}) can be written, in our case, as
  $\exp(\Psi_N(u))$ where
\begin{align}
\hskip-3mm  \Psi_N(u)& =\ii u(b_N-\beta_N)-\frac{u^2}{2}(\sigma_N-\gamma_N)
  +\int_{\Rr} (e^{\ii u x}-1-\ii u h(x))\nu_N(\dd x) \label{in1}\\
  & = \ii u(b_N-\beta_N)-\frac{u^2}{2}(\sigma_N+\nu_N(h^2)-\gamma_N)-
  \ii\frac{u^3}{6}\kappa_N+\nu_{N}(k_u), \label{in2}
\end{align}
where
\begin{equation}\label{in3}
k_u(x)=e^{\ii u x}-1-\ii u h(x)
+\dfrac{u^2 h(x)^2}{2}+\ii\dfrac{u^3 h(x)^3}{6}.
\end{equation}
\par
(a) First assume that Conditions (1) and (2) above hold, and define
$\beta_N$ and $\gamma_N$ by~(\ref{eq-select}). Since $h$ is bounded
with compact support and $h(x)=x$ in a neighborhood of $0$, it is
easily seen that $|k_u(x)|\leq C_u (x^4\wedge1)$, where $C_u$ is a
constant depending on $u$. It thus follows from (\ref{in2}),
(\ref{eq-select}) and Conditions (1) and (2) that
$\Psi_N(u)\to\Psi(u)$, as defined by~(\ref{rep:2}).  Thus, it only
remains to prove that we have $\nu_N(k_u)\to\nu(k_u)$ uniformly (with
respect to $u$) on all compact sets $K\subset\RR$. For this, we use
fairly standard arguments.

Let $K$ be a compact subset of $\RR$, and let $\eps>0$ be
fixed. For any $A>1$, let
$$
g_A(x)=\begin{cases}
0&\text{ if } -|x|<A\\
\frac{|x|}{A}-1&\text{ if } A\leq |x|\leq 2A\\
1&\text{ if } |x|>A,
\end{cases}
$$
and $h_A=1-g_A$. The function $g_A$ is continuous with $0\leq g_A\leq
x^4\wedge 1$, and $g_A\ra 0$ pointwise as $A\to\infty$. Therefore we
have the limits
\begin{equation}\label{eq-limits}
\lim_{n\ra +\infty}{\nu_n(g_A)}=\nu(g_A),
\quad \lim_{A\ra +\infty}{\nu(g_A)}=0.
\end{equation}
\par
Now write
$$
|\nu_n(k_u)-\nu(k_u)|\leq |\nu_n(g_Ak_u)|+|\nu(g_Ak_u)|+
|\nu_n(h_Ak_u)-\nu(h_Ak_u)|.
$$
\par
There exists a constant $C_K\geq 0$, depending only on $K$, such that 
$$
|k_u(x)g_A(x)|\leq C_Kg_A(x)
$$
for all $u\in \RR$, $x\in \RR$. Consequently, combining the two limits
in~(\ref{eq-limits}), we see that there exists $N_0\geq 1$ and $B>0$
such that, for any $N\geq N_0$, we have
\begin{equation}\label{eq-ga}
|\nu_n(g_Bk_u)|+|\nu(g_Bk_u)|\leq 2\eps,
\end{equation}
uniformly for $u\in \Rr$.
\par
The other term, where now $A=B$ is fixed, is also easily dealt
with. First, $k_uh_B$ is continuous and satisfies $|k_u(x)h_B(x)|\leq
C_u(x^4\wedge1)$, for some $C_u$ depending on $u$, so that by
assumption we have
$$
\lim_{N\ra +\infty}{\nu_N(k_u h_B)}=\nu(k_uh_B),
$$ 
for any $u\in\RR$. In addition, for each fixed $x$, the map
$$
u\mapsto k_u(x)h_B(x)
$$
is differentiable and its derivative is bounded by $D_{K}(x^4\wedge1)$
for $x\in\RR$ and $u\in K$, the constant $D_{K}$ depending only on
$K$. Consequently, the function $u\mapsto\nu_N(k_u h_B)$ is
differentiable on $\Rr$, and its derivative is uniformly bounded on
the compact set $K$. It follows that the family of functions
$(u\mapsto \nu_N(k_uh_B))_N$ is equicontinuous on $K$, and therefore
its pointwise convergence to $\nu(k_uh_B)$ is uniform on $K$. Thus for
some $N_1$, we have
$$
|\nu_n(h_Bk_u)-\nu(h_Bk_u)|\leq \eps
$$
for all $N\geq N_1$ and $u\in K$.  From this and the previous
estimate~(\ref{eq-ga}), the uniform convergence on compact sets
follows.

(b) Conversely, we assume that (\ref{M}) holds, locally uniformly in
$u$, with some parameters $(\beta_N,\gamma_N)$ and limiting
function $\Phi$. Since $\Phi$ is continuous and non-vanishing on
some neighborhood $I=(-2\delta,2\delta)$ of $0$ and $\Phi(0)=1$,
it follows from basic results of complex analysis (see,
e.g., \cite[Lemmas 7.6 and 7.7]{Sato}) that on $I$ we
have $\Phi=\exp(\Psi)$, where $\Psi$ is
continuous with $\Psi(0)=0$, and moreover $\Psi_N$, as defined by
(\ref{in2}), converges to $\Psi$ uniformly on $I$.

We deduce that, if $\theta$ is a locally bounded function on $\Rr$, we
have
\begin{eqnarray}
A_N(u)&=& \int_{-\delta}^{\delta}\left(\Psi_N(u+y)-
\Psi_N(u)\right)\vartheta(y)dy\nonumber\\
&&\to~ A(u)~=~\int_{-\delta}^{\delta}\left(\Psi(u+y)-
\Psi(u)\right)\vartheta(y)dy
\label{B1}
\end{eqnarray}
for all $u$ such that $|u|<\delta$.

Now, we observe that the
orthogonality properties
 \begin{equation}
\int_{-\delta}^{\delta} \vartheta(y)y dy=\int_{-\delta}^\delta
{\vartheta(y)y^2dy}=0
\quad \text{ for }\quad \vartheta(y)=\frac{\delta^7}{105}
~(5y^2-3\delta^2),
\end{equation}
together with Fubini's Theorem allow us to eliminate the terms
involving $h(x)$, $\sigma_N$ and $b_N$ in~(\ref{in1}), to yield
 \begin{equation}\label{r2}
A_N(u)~=~\int_{\Rr} e^{\ii ux}g(x)\nu_N(\dd x),
\end{equation}
where
\begin{eqnarray*}
g(x)&=&
\int_{-\delta}^\delta{\vartheta(y)(e^{ixy}-1)dy}\\
&=&\frac{\delta^7}{105}\left(\frac{8\delta^3}3
+\frac{4\delta^2\sin(\delta x)}x+\frac{20\delta\cos(\delta x)}{x^2}
-\frac{20\sin(\delta x)}{x^3}\right).
\end{eqnarray*}
\par
The function $g$ is continuous and, as checked by straightforward
calculus, satisfies
\begin{equation}\label{eq-a}
g(x)\sim x^4~~\text{ as } x\ra 0,
\quad
C^{-1}(x^4\wedge 1)\leq g(x)\leq C(x^4\wedge 1)\quad
\forall x\in\Rr,
\end{equation}
for some constant $C>0$ (in particular, $g(x)=0$ if and only
if $x=0$). Note that $A_N$ is the Fourier transform of the positive
finite measure $\mu_N(dx)=g(x)\nu_N(dx)$, and the convergence
(\ref{B1}) for all $u$ with $|u|\leq\delta$ implies (see for example
the proof of Theorem 19.1 in \cite{jacodprotter}) that the sequence of
measures $\mu_N$ is relatively compact for the weak convergence (or,
tight).

In particular, there is a subsequence $(\mu_{N_k})$ which
converges weakly to a positive finite measure $\mu$. Then obviously
Condition (2) is satisfied by the sequence $(\nu_{N_k})$, with the
limiting measure $\nu$ defined by
$$
\nu(dx)~=~\dfrac{1}{g(x)}\I_{\RR\backslash\{0\}}\,\mu(dx).
$$
\par
Next, we check that Condition (1) is satisfied by the subsequence
$(\kappa_{N_k})$. To this end, we observe that the imaginary part of
$\Psi_N(u)$, as given by~(\ref{in2}), can be written as follows:
\begin{align*}
\Imag(\Psi_N(u))&=
u(b_N-\beta_N)+\int_{\Rr}(\sin(ux)-uh(x))\nu_N(\dd x)\\
&=u(b_N-\beta_N)
-\dfrac{u^3}{6}~\kappa_N+\nu_N(\ell_u),
\end{align*}
where 
$$
\ell_u(x)=\sin(ux)-uh(x)+\dfrac{u^3h(x)^3}{6}.
$$ 
\par
Since $\Psi_N(u)\to\Psi(u)$, we have
$\Imag(\Psi_N(u))\to\Imag(\Psi(u))$ for all $u\in I$. The
function $\ell_u$ is continuous and also satisfies $|\ell_u(x)|\leq
C_u(x^4\wedge1)$, for some constant $C_u\geq 0$, and hence
$\nu_{N_k}(\ell_u)\to\nu(\ell_u)$. Consequently,
$$
u(b_{N_k}-\beta_{N_k})-\frac{u^3}{6}~\kappa_{N_k}~\to
\Imag(\Psi(u))-\nu(\ell_u)
$$
as $k\ra+\infty$, and for $u\in I$. Then obviously
$\Imag(\Psi(u))-\nu(\ell_u)=au+bu^3$ for some $a$, $b\in\RR$, and it
follows that $\kappa_{N_k}\to
\kappa=6b$. This proves Condition (1) for $(\kappa_{N_k})$.

To conclude the proof, we apply the sufficient condition (a) of our
theorem to the subsequence $(Z_{N_k})$, which implies
that $\Phi=e^\Psi$, where $\Psi$ is given by (\ref{rep:2}). Therefore
$\Phi$ does not vanish, and henceforth we can take $I=\Rr$ in the
previous proof. We deduce that the convergence (\ref{B1}) holds for
all $u\in\Rr$, and then L\'evy's Theorem yields that not only is the
sequence $(\mu_N)$ tight, but it actually converges to a limit
$\mu$. Therefore the previous proof holds for the original sequences
$(\nu_N)$ and $(\kappa_N)$, and we are done.
\end{proof}

We now state as a corollary a weak limit theorem for variables
satisfying the assumptions of Theorem \ref{thm:3}. Of course, (1)
below is a classical result.

\begin{cor}
  Let $Z_N$ be a sequence of infinitely divisible random variables
  satisfying one of the equivalent conditions of Theorem \ref{thm:3},
  with generating triplets $(0,0,\nu_N)$. Then we have:
\par
\emph{(1)} If $\nu_N(h^2)\to\widetilde{\sigma}\in[0,+\infty[$, then
the sequence $(Z_N)$ converges in law to a limit random variable $Z$,
which is necessarily infinitely divisible, with generating triplet
$(0,\sigma,\nu)$ with $\nu$ as in \emph{(\ref{rep:2})},
but in this case $x^2\wedge1$ is integrable with respect to $\nu$,
and $\sigma=\widetilde{\sigma}-\nu(h^2)$.
\par
\emph{(2)} If $\nu_N(h^2)\to+\infty$, then $Z_N/\sqrt{\nu_N(h^2)}$
converges in law to the standard Gaussian random variable
$\mathcal{N}(0,1)$.
\end{cor}

\begin{proof}
  The results follow from the fact that under our assumptions, we have 
$$
e^{u^2\nu_N(h^2)/2}\E[e^{\ii u Z_N}]~\to~\Phi(u)
$$
locally uniformly for $u\in\Rr$.
\end{proof}

\section{Some examples of mod-Gaussian convergence in
  arithmetic}\label{sec-nt}

In this section, which is largely independent of the previous one, we
give two examples of (unconditional) instances of mod-Gaussian
convergence in analytic number theory. The first is quite elementary
and formal. For the second (involving function fields), we again
summarize briefly the required information to understand the
statements, this time for probabilist readers. In addition,
Section~\ref{ssec-poisson} explains how to interpret the Erd\H{o}s-K\'ac
theorem in terms of mod-Poisson convergence.

\subsection{The arithmetic factor in the moment conjecture for
  $\zeta(1/2+it)$}\label{ssec-arith}

We come back to the moment conjecture~(\ref{momentszeta}) for the
Riemann zeta function, which we recall: we should have
$$
\lim_{T\ra +\infty}{
  \frac{1}{T(\log T)^{\lambda^2}}
\int_0^{T}{|\zeta(\demi+it)|^{2\lambda}dt}}
=A(\lambda)M(\lambda)
$$
for any complex number $\lambda$ such that $\Reel(\lambda)>-1$, where
\begin{align}
M(\lambda)&=
\frac{G(1+\lambda)^2}{G(1+2\lambda)},\quad\quad
\text{$G(z)$ the Barnes double-gamma function},
\\
A(\lambda)&=\prod_{p}{
  \Bigl(1-\frac{1}{p}\Bigr)^{\lambda^2}\Bigl\{
\sum_{m\geq 0}{\Bigl(\frac{\Gamma(m+\lambda)}{m!\Gamma(\lambda)}\Bigr)^2
p^{-m}}\Bigr\}}.
\end{align}
\par
From~(\ref{ML}), it follows that the random matrix factor $M(iu)$, for
$u\in \Rr$, occurs as the limiting function for the mod-Gaussian
convergence of some natural sequence of random variables (namely,
characteristic polynomials of random unitary matrices of growing size
distributed according to Haar measure).  It is thus natural to wonder
whether the arithmetic factor has the same property, since, if that
were the case, the formal properties of mod-Gaussian convergence
(specifically, (3) in Proposition~\ref{prop-formal}) imply that there
exists a sequence of random variables which converges in mod-Gaussian
sense with limiting function $A(iu)M(iu)$. We will show that this is
the case (indeed with strong mod-Gaussian convergence); we believe
this structure of the moments has arithmetic significance, although
the computation we do now (though enlightening) does not yet explain
this.

\begin{prop}
There exists a sequence $(Z_N)$ of positive real-valued random
variables and positive real numbers $\gamma_N>0$ such that
$$
e^{u^2\gamma_N /2}\E(e^{iuZ_N})\ra A(iu)
$$
locally uniformly for $u\in \Rr$. 
\end{prop}

\begin{proof}
We start by writing $A(iu)$ as a limit
$$
A(iu)=\lim_{N\ra +\infty}{A_1(u,N)A_2(u,N)}
$$
where
\begin{gather}
A_1(u,N)=\prod_{p\leq N}{(1-p^{-1})^{-u^2}},\\
A_2(u,N)=\prod_{p\leq N}{
\sum_{m\geq 0}{\Bigl(\frac{\Gamma(m+iu)}{m!\Gamma(iu)}\Bigr)^2
p^{-m}}
}=\prod_{p\leq N}{\hyperg{iu}{iu}{1}{p^{-1}}},
\end{gather}
by the definition of the Gauss hypergeometric function
$$
\hyperg{a}{b}{c}{z}=\sum_{k\geq 0}{
\frac{a(a+1)\cdots (a+k-1)b(b+1)\cdots (b+k-1)}
{c(c+1)\cdots (c+k-1)}\frac{z^k}{k!}}.
$$
\par
We now recall that
$$
\prod_{p\leq N}{(1-p^{-1})}\sim e^{-\gamma}(\log N)^{-1}
$$
as $N\ra +\infty$, by the Mertens formula (see,
e.g.,~\cite[Th. 429]{hardy-wright}). Thus, it is natural to consider
$$
\gamma_N=2(\gamma+\log\log N)>0,\quad\quad N\geq 2,
$$
because we then have
$$
\lim_{N\ra +\infty}{e^{u^2 \gamma_N /2}A_2(u,N)}
=A(iu).
$$
\par
It is then enough to show that, for any $N\geq 2$, the factor
$A_2(u,N)$ is the characteristic function $\E(e^{iuZ_N})$ of a random
variable $Z_N$ to deduce
$$
\lim_{N\ra +\infty}{e^{u^2 \gamma_N /2}\expect(e^{iuZ_N})}
=A(iu).
$$
\par
Furthermore, since $A_2(u,N)$ is defined as a product, it is enough to
show that each hypergeometric factor $\hyperg{iu}{iu}{1}{p^{-1}}$, $p$
prime, is the characteristic function of a random variable $X_p$ to
obtain the desired result with $Z_N$ the sum of independent variables
distributed as $X_p$ for $p\leq N$.  Lemma~\ref{lm-hyperg} below,
applied with $x=p$, checks that this is the case. Then, finally, the
convergence is locally uniform because so is the convergence of the
Euler product defining $A(iu)$.
\end{proof}

\begin{lem}\label{lm-hyperg}
Let $X$ be a complex-valued random variable uniformly distributed over
the unit circle, and let $x$ be a real number with $x>1$. Then
we have
$$
\expect(e^{iu(\log |1-x^{-1/2}X|^{-2})})=
\expect\Bigl(|1-x^{-1/2}X|^{-2iu}\Bigr)=\hyperg{iu}{iu}{1}{x^{-1}}.
$$
\end{lem}

\begin{proof}
With $X$ as described, we have
$$
\Bigl|1-\frac{X}{\sqrt{x}}\Bigr|^2=
1+\frac{1}{x}-\frac{2\Reel(X)}{\sqrt{x}},
$$
which is always $\geq (1-x^{-1/2})^2>0$. Since $\Reel(X)$ is
distributed like $\cos\Theta$, where $\Theta$ is uniformly distributed
on $[0,2\pi]$, we have
$$
\expect(e^{iu\log |1-x^{-1/2}X|^{-2}})=
\frac{1}{2\pi}
\int_{0}^{2\pi}{
(1+x^{-1}-2x^{-1/2}\cos\theta)^{-iu}d\theta
}.
$$
\par
Now it is enough to apply~\cite[9.112]{gr} (with $n=0$, $p=iu$,
$z=x^{-1/2}$) to see that this expression is exactly
$\hyperg{iu}{iu}{1}{x^{-1}}$. 
\end{proof}

\begin{rem}
In view of~(\ref{momentszeta}) and the Euler product (formal)
expansion
$$
|\zeta(\demi+it)|^2~~ \text{``}=\text{''}~~
\prod_{p}{\Bigl|1-\frac{1}{p^{1/2 +it}}\Bigr|^{-2}},
$$
the mod-Gaussian convergence of the arithmetic factor is described
concretely as follows: let $(X_p)$ be a sequence of independent random
variables identically and uniformly distributed on the unit
circle. Then the sequence of random variables defined by
$$
\sum_{p\leq N}{
\log \Bigl|1-\frac{X_p}{\sqrt{p}}\Bigr|^{-2}
}=
\log \prod_{p\leq N}{\Bigl|1-\frac{X_p}{\sqrt{p}}\Bigr|^{-2}}
$$
converges as $N\ra +\infty$, in the mod-Gaussian sense, with limiting
function given by the arithmetic factor for the moments of
$|\zeta(\demi +it)|^2$ in~(\ref{momentszeta}), evaluated at $iu$, and
parameters $(0,2\log( e^{\gamma}\log N))$.
\end{rem}

\begin{rem}
  The computations above are implicit in some earlier work, for
  instance in the paper~\cite{cfkrs} of Conrey, Farmer, Keating,
  Rubinstein and Snaith on refined conjectures for moments of
  $L$-functions.
\end{rem}

\subsection{Some families of $L$-functions over function fields}
\label{ssec-function-fields}

In this section, we give an example of mod-Gaussian convergence
in the setting of families of $L$-functions, as developed by Katz and
Sarnak~\cite{katzsarnak}.  We do not try to summarize the most general
context in which they operate, in order to keep prerequisites from
algebraic geometry to a minimum, concentrating on one concrete example
which is already of great interest and can be explained ``from
scratch''. It is the family of hyperelliptic curves, which is
described in~\cite[\S 10.1.18, 10.8]{katzsarnak}.
\par
The fundamental result linking families of $L$-functions and Random
Matrix Theory is the equidistribution theorem of Deligne, which we
phrase (in the special case under consideration) in more probabilistic
language than usual to clarify its meaning for probabilists. Its
content is then that some sequences of random variables, defined
arithmetically and taking values in the \emph{set of conjugacy
  classes} in compact Lie groups such as $U(N)$, converge in law to
the image on the space of conjugacy classes of the probability Haar
measure on the group. Consequently, the values of the characteristic
polynomials of such random variables are approximately distributed
like the variables $Z_N$ in the Keating-Snaith limit
formula~(\ref{momentinfini}) for suitable values of $N$ (or their
analogues for other groups).
\par
Let $p$ be an odd prime number and let $q=p^n$, $n\geq 1$, be a power
of $p$. We denote by $\Fp_q$ a field with $q$ elements, in particular
$\Fp_p=\Zz/p\Zz$. Recall from the theory of finite fields that if we
fix an algebraic closure $\bar{\Fp}_q$ of $\Fp_q$, then for every
$n\geq 1$ there exists a unique subfield $\Fp_{q^n}$ of $\bar{\Fp}_q$
which has order $q^n$ (i.e., it is a field extension of degree $n$ of
$\Fp_q$), which is characterized as the set of $x\in\bar{\Fp}_q$ such
that $x^{q^n}=x$.
\par
Let $g\geq 1$ be an integer, and let $f\in \Fp_q[T]$ be a monic
polynomial of degree $2g+1$ with no repeated roots (in an algebraic
closure $\bar{\Fp}_q$ of $\Fp_q$). Then the set $C_f$ of solutions, in
$\bar{\Fp}_q^2$, of the polynomial equation
$$
C_f\,:\, y^2=f(x)=x^{2g+1}+a_{2g}x^{2g}+\cdots +a_1x+a_0,
\quad\quad\text{(say)}.
$$
is called an \emph{affine hyperelliptic curve of genus $g$}. Taking
the associated homogeneous equation in projective coordinates
$[x:y:z]$, one gets the projective curve
$$
\tilde{C}_f\,:\,y^2z^{2g-1}=f(xz^{-1})z^{2g+1}=x^{2g+1}+a_{2g}zx^{2g}+\cdots
+a_1z^{2g}x+a_0z^{2g+1},
$$
which is still smooth and corresponds to $C_f$ with an added point at
infinity with projective coordinates $[0:1:0]$. 
\par
For every $n\geq 1$, denote by $\tilde{C}_f(\Fp_{q^n})$ the set of
points in $\tilde{C}_f$ which have coordinates in the subfield
$\Fp_{q^n}$ of $\bar{\Fp}_q$ (note that $\tilde{C}_f(\Fp_{q^n})\simeq
C_f(\Fp_{q^n})\cup \{[0:1:0\}$). The $L$-function $P_f(T)$ of
$\tilde{C}_f$ (sometimes called the $L$-function of $C_f$ instead) is
then defined as the numerator of the zeta function $Z(\tilde{C}_f)$
defined by the formal power series expansion
$$
Z(\tilde{C}_f)=\exp\Bigl(\sum_{n\geq
  1}{\frac{|\tilde{C}_f(\Fp_{q^n})|}{n}T^n}\Bigr)=
\exp\Bigl(\sum_{n\geq 1}{\frac{|C_f(\Fp_{q^n})|+1}{n}T^n}\Bigr),
$$
which is known to represent a rational function of the form
$$
Z(\tilde{C}_f)=\frac{P_f(T)}{(1-T)(1-qT)},
$$
which determines uniquely the $L$-function $P_f$.
\par
The following properties of $P_f$ were all proved by 1945, in
particular thanks to the work of A. Weil on the Riemann Hypothesis for
curves over finite fields:
\par
\noindent -- $P_f$ is a polynomial with integer coefficients of degree
$2g$, with $P_f(0)=1$.
\par
\noindent -- [Functional equation (F.K. Schmidt)] We have the
polynomial identity
$$
q^gT^{2g}P_f\Bigl(\frac{1}{qT}\Bigr)=P_f(T).
$$
\par
\noindent -- [Riemann Hypothesis (A. Weil)] If we write
\begin{equation}\label{eq-factor}
P_f(T)=\prod_{1\leq j\leq 2g}{(1-\alpha_{f,j}T)},\quad\quad
\alpha_{f,j}\in\CC,
\end{equation}
then all the inverse roots $\alpha_{f,j}$ satisfy
$|\alpha_{f,j}|=\sqrt{q}$. 
\par
The following property, which Weil could already prove in some form,
is much better understood in the framework of algebraic geometry as
developed in the 1960's by the Grothendieck school:
\par
\noindent -- [Spectral interpretation] There exists a well-defined
conjugacy class $F_f$ in the set $U(2g,\Cc)^{\sharp}$ of conjugacy
classes in the compact unitary group $U(2g,\Cc)$ such that
\begin{equation}\label{eq-spectral}
P_f(T)=\det(1-q^{1/2}TF_f)
\end{equation}
(this conjugacy class is the unitarized geometric Frobenius conjugacy
class of $\tilde{C}_f$). Moreover, there exists a non-degenerate
alternating form $\langle\cdot,\cdot\rangle$ on $\Cc^{2g}$ such that
$F_f$ is a conjugacy class in $USp(2g,\langle\cdot,\cdot\rangle)$, the
unitary symplectic group of matrices which leave the alternating form
invariant. Because any two non-degenerate alternating forms are
conjugate over $\CC$, we can (and will) see $F_f$ as a well-defined
conjugacy class in $USp(2g,\CC)$, the unitary symplectic group for the
standard symplectic form.

\begin{rem}
  If one makes the substitution $T=q^{-s}$, $s\in \CC$, to define a
  complex-variable $L$-function 
\begin{equation}\label{eq-lf}
L(f,s)=P_f(q^{-s}),
\end{equation}
the functional equation and Riemann Hypothesis become exact analogues
of the corresponding property and conjecture for the Riemann zeta
function and its zeros, namely
$$
L(f,s)=q^{2g(1/2-s)}L(f,1-s),
$$
and all zeros of $L(f,s)$ have real part equal to $\demi$.
\par
The spectral interpretation, on the other hand, which implies that
zeros of the $L$-function are eigenvalues of a unitary matrix (defined
only up to conjugation), is much more mysterious for the Riemann zeta
function. Note that the fact that $F_f$ is in fact a symplectic
conjugacy class implies the functional equation, by simple linear
algebra.
\end{rem}

Here is now one particular case of the Deligne Equidistribution
Theorem. 

\begin{thm}[Katz-Sarnak]\label{th-deligne}
  Fix an integer $g\geq 1$. For every power $q$ of an odd prime $p$,
  let $\mathcal{H}_{g,q}$ be the set of monic polynomials in
  $\Fp_q[T]$ of degree $2g+1$ which have no multiple roots.  Let
  $H_{g,q}$ be random variables with values in
  $USp(2g,\CC)^{\sharp}$ and with distributions given by
\begin{equation}\label{eq-deligne-discr}
\Prob(H_{g,q}=C)=\frac{1}{|\mathcal{H}_{g,q}|}
|\{f\in \mathcal{H}_{g,q}\,\mid\, F_f=C
\}|
\end{equation}
for any conjugacy class $C\in USp(2g,\CC)^{\sharp}$. 
\par
Then, as $q\ra +\infty$, among odd powers of primes, the random
variables $H_{g,q}$ converge in law to a random variable $H_g$
distributed according to
\begin{equation}\label{eq-deligne-cont}
\Prob(H_g\in A)=\mu_g(A),
\end{equation}
for any conjugacy-invariant measurable set $A$ in $USp(2g,\CC)$, where
$\mu_g$ is the probability Haar measure on $USp(2g,\CC)$. In other
words, $H_g$ is distributed according to the image on the space of
conjugacy classes of the probability Haar measure on $USp(2g,\CC)$.
\end{thm}

\begin{proof}
  This is exactly Theorem 10.8.2 of Katz and
  Sarnak~\cite{katzsarnak},\footnote{\ Except that there is a typo in
    their statement, namely in the right-hand side of line 13,
    $\mu(intrin,g,k_i,\alpha_{k_i})$ should be replaced by
    $\mu(hyp,d,g,k_i,\alpha_{k_i})$ and line 14 can be deleted.}
  taking the choice of $\alpha_k$ to be $\sqrt{|k|}$ so that the
  conjugacy class denoted $\vartheta(k,\alpha_k,C_f/k)$ is the same as
  the class $F_f$ for $f\in \mathcal{H}_{g,|k|}$ (see the discussion
  in 10.7.2 of~\cite{katzsarnak}), and the
  distribution~(\ref{eq-deligne-discr}) of $H_{g,q}$ is the same as
  the measure $\mu(hyp,2g+1,g,\Fp_q,\sqrt{q})$ as defined in 10.8.1
  of~\cite{katzsarnak}.
\end{proof}

\begin{rem}
  This result is derived by an application of Deligne's
  Equidistribution Theorem.  Deligne's theorem is much more general:
  in fact, for \emph{any} ``algebraic'' family of $L$-functions (a
  much more general notion that what we have described) there is
  always an equidistribution theorem which can be interpreted as
  convergence in law of random variables defined similarly
  to~(\ref{eq-deligne-discr}) for some conjugacy classes associated
  with the family, to a random variable distributed according to the
  image of the probability Haar measure on a group which can be
  interpreted as ``the smallest group for which equidistribution may
  conceivably hold'' (see~\cite[\S 9.2, 9.3, 9.7]{katzsarnak} for
  detailed discussions of various versions). Much of the work in
  applying Deligne's Equidistribution Theorem is concentrated in the
  determination of this group (often called the geometric monodromy
  group of the family). For the case of the family given by the
  $\mathcal{H}_{g,q}$ in Theorem~\ref{th-deligne}, the content of the
  result is that this monodromy group is the whole symplectic group
  $Sp(2g)$, and this is proved in~\cite[Th. 10.1.18.3]{katzsarnak}.
\end{rem}

The following proposition is then immediate.

\begin{prop}\label{pr-mellin}
  Let $g\geq 1$ be an integer. Let $H_{g,q}$ be random variables as in
  Theorem~\ref{th-deligne}. For any $\lambda \in\CC$ with
  $\Reel(\lambda)>0$, we have
$$
\lim_{g\ra +\infty}{\lim_{q\ra
    +\infty}{\frac{1}{g^{(\lambda^2+\lambda)/2}}
    \E(\det(1-H_{g,q})^{\lambda})}}=M_{Sp}(\lambda),
$$
where
$$
M_{Sp}(\lambda)=
2^{-\lambda^2/2}\Bigl(\frac{\pi}{2}\Bigr)^{\lambda/2} 
\frac{G(3/2)}{G(3/2+\lambda)}
$$
\par
In particular, for any integer $k\geq 1$, we have
$$
\lim_{g\ra +\infty}{\lim_{q\ra
    +\infty}{\frac{1}{g^{(k^2+k)/2}}
    \E(\det(1-H_{g,q})^{k})}}=
\prod_{j=1}^k{\frac{1}{(2j-1)!!}
}.
$$
\end{prop}

\begin{proof}
  For $\Reel(\lambda)>0$, the function
$$
x\mapsto \det(1-x)^{\lambda}
$$
is continuous and bounded on $USp(2g,\Cc)^{\sharp}$; this is because,
in terms of the eigenvalues $e^{i\theta_j}$ of $x\in USp(2g,\CC)$
arranged in pairs so that $\theta_{2g+1-j}=-\theta_j$, we
have\footnote{\ This is an analogue of the non-negativity of the
  central special value $L(f,1/2)$ for any self-dual $L$-function,
  which is implied by the Riemann Hypothesis.}
$$
\det(1-x)=\prod_{1\leq j\leq g}{|(1-e^{i\theta_j})|^2}\geq 0.
$$
\par
Keating and Snaith~\cite[eq. (26) \& (32)]{keating-snaith2} have shown
that
\begin{equation}\label{eq-ks-sp}
  \lim_{g\ra
    +\infty}{g^{-(\lambda^2+\lambda)/2}\E(\det(1-H_g)^{\lambda})}=
  2^{\lambda^2/2}\frac{G(1+\lambda)\sqrt{\Gamma(1+\lambda)}}
{\sqrt{G(1+2\lambda)\Gamma(1+2\lambda)}}
\end{equation}
for any complex number $\lambda$ with $\Reel(\lambda)>-1$, and
therefore the statement is a direct consequence of convergence in law
of $H_{g,q}$ to $H_g$ once we prove that our expression for
$M_{Sp}(\lambda)$ coincides with this value. This, however, is easy to
check: using the duplication formula for $G(z)$ in the form
$$
G(\demi)^2G(2z)=
(2\pi)^{-z}2^{2z^2-2z+1}\Gamma(z)(G(z)G(z+\demi))^2
$$
one can rewrite the function in the square-root of the denominator of
the Keating-Snaith expression as
\begin{multline*}
\Gamma(1+2\lambda)G(1+2\lambda)=G(2(1+\lambda))
= 
(2\pi)^{-(\lambda+1)}2^{2(\lambda+1)^2-2(\lambda+1)+1}\\
\Gamma(1+\lambda)
\Bigl(\frac{G(1+\lambda)G(3/2+\lambda)}{G(\demi)}\Bigr)^2,
\end{multline*}
and deduce the result after a short computation.
\par
The last expression for $M_{Sp}(k)$ is equation (34)
in~\cite{keating-snaith2} and follows also from the definition of
$M_{Sp}(\lambda)$ using $G(z+1)=\Gamma(z)G(z)$ and the formula 
$$
\Gamma(k+\demi)=\frac{(2k-1)!!}{2^k}\sqrt{\pi},
$$
for $k\geq 0$ integer.
\end{proof}

We can not argue quite so quickly to derive a mod-Gaussian convergence
result because the random variables $\log \det(1-H_{g,q})$ are not
defined whenever $H_{g,q}$ has an eigenvalue $1$, and this can occur
with positive probability. Indeed, by~(\ref{eq-spectral})
and~(\ref{eq-factor}), we have
$$
\det(1-F_f)=P_f(q^{-1/2})=\prod_{1\leq j\leq
  2g}{(1-q^{-1/2}\alpha_{f,j})}
$$
for any $f\in\mathcal{H}_{g,q}$, and so the issue is whether
$\sqrt{q}$ is a zero of $Z(C_f)$ (equivalently, whether $L(f,1/2)=0$),
and this may well happen (e.g., if $q=p^2$ with $p\equiv 3\mods{4}$,
for the curve $E\,:\,y^2=x^3-x$, it is well-known, and easy to show,
that $P_E(T)=(1-pT)^2$).
\par
Nevertheless, it is not too hard to prove the following:

\begin{prop}\label{pr-mod-gaussian-finite-fields}
  Let $g\geq 1$ be an integer. For any power $q\not=1$ of an odd prime
  $p$, let $\tilde{\mathcal{H}}_{g,q}$ be the subset of those $f\in
  \mathcal{H}_{g,q}$ such that $L(f,1/2)=P_f(q^{-1/2})\not=0$.
\par
\emph{(1)} Let $I_{g,q}$ be random variables with values in
$USp(2g,\CC)^{\sharp}$ such that
$$
\proba(I_{g,q}=C)=\frac{1}{|\tilde{\mathcal{H}}_{g,q}|}
|\{
f\in \tilde{\mathcal{H}}_{g,q}\,\mid\, F_f=C
\}|
$$
for any $C\in USp(2g,\CC)^{\sharp}$. Then $I_{g,q}$ converges in law to
$H_g$ as $q\ra+\infty$.
\par
\emph{(2)} Let $L_{g,q}=\log \det(I-I_{g,q})$ which is a well-defined
real-valued random variable. We have the mod-Gaussian convergence
$$
\lim_{g\ra +\infty}{\lim_{q\ra +\infty}{
g^{-iu/2+u^2/2}    \E(e^{iuL_{g,q}})}}=M_{Sp}(iu),
$$
for any $u\in\RR$.
\end{prop}

Note that we have here mod-Gaussian convergence with parameters
given by $(\demi \log g,\log g)$.

\begin{proof}
(1) If $\varphi$ is a bounded continuous function on
$USp(2g,\CC)^{\sharp}$, we have
$$
|\E(\varphi(I_{g,q}))-\E(\varphi(H_{g,q}))|\leq
\|\varphi\|_{\infty} 
\proba(\det(1-H_{g,q})=0)=\|\varphi\|_{\infty}
\proba(H_{g,q}\in A_g),
$$
where $A_g=\{x\in USp(2g,\CC)^{\sharp}\,\mid\, \det(1-x)=0\}$. This
set $A_g$ is a closed set with empty interior, hence boundary equal to
$A_g$, which has Haar measure zero. By Theorem~\ref{th-deligne} and
the standard properties of convergence in law, we have
$$
\lim_{q\ra+\infty}{\proba(H_{g,q}\in A_g)}=\proba(H_g\in A_g)=0,
$$
and it follows then that
$$
\lim_{q\ra +\infty}\E(\varphi(I_{g,q}))=\E(\varphi(H_{g,q})),
$$
which justifies the convergence in law of $I_{g,q}$.
\par
(2) For $f\in \tilde{\mathcal{H}}_{g,q}$, we have $\det(1-F_f)>0$, and
therefore the definition of the law of $I_{g,q}$ shows that $L_{g,q}$
is well-defined. Because of the Keating-Snaith limit
formula~(\ref{eq-ks-sp}), valid for all complex numbers with
$\Reel(\lambda)>-1$, it is enough to show that, for all $u\in\RR$, we
have
$$
\lim_{q\ra +\infty}{\E(e^{iuL_{g,q}})}=
\E(\det(1-H_g)^{iu}).
$$
\par
The function $\varphi$ on $USp(2g,\CC)^{\sharp}$ defined by
$$
x\mapsto  
\begin{cases}
0,&\text{ if } \det(1-x)=0\\
\det(1-x)^{iu},&\text{ otherwise},
\end{cases}
$$
is bounded and its set of points of discontinuity is the set $A_g$ of
Haar-measure $0$. By a fairly standard result on convergence in law,
this and the convergence in law of $I_{g,q}$ to $H_g$ suffice to
ensure that
$$
\lim_{q\ra +\infty} \E(\varphi(I_{g,q}))=\E(\varphi(H_g))
$$
(see, e.g.,~\cite[Ch. 4, \S 5, n${{}^o}$12, Prop. 22]{bourbaki},
properly translated, or one can of course do the necessary $\eps$
management by hand). By definition, the left-hand side is
$\E(e^{iuL_{g,q}})$, while the right-hand side is
$\E(\det(1-H_g)^{iu})$ (since $A_g$ has measure zero, once more).
\end{proof}

\begin{rem}
  Although we used the example of the family $\mathcal{H}_{g,q}$ in
  this section, it is clear from the proofs that the argument goes
  through with no change for any family with symplectic monodromy, and
  that suitable analogues will hold for families with unitary or (with
  a bit more care because of the issue of forced vanishing at the
  critical point) with orthogonal symmetry.
\end{rem}

\begin{rem}
  In terms of $L$-functions as defined in~(\ref{eq-lf}), we can
  rephrase the last limit as follows:
$$
\lim_{g\ra +\infty}
\lim_{q\ra +\infty}
\frac{g^{u^2/2}}{|\tilde{\mathcal{H}}_{g,q}|}
\sum_{f\in \tilde{\mathcal{H}}_{g,q}}{
\Bigl(\frac{L(f,1/2)}{\sqrt{g}}\Bigr)^{iu}
}=M_{Sp}(iu).
$$
\par
It remains a big problem to obtain results of this type without the
inner limit over $q$, which already transforms the arithmetic to a
pure ``random matrix'' problem by the magic of Deligne's
equidistribution theorem (see the comments and conjectures
in~\cite[p. 12, 13]{katzsarnak}, in particular Example (2),
p. 13). However, Faifman and Rudnick~\cite{fr} and Kurlberg and
Rudnick~\cite{kr} have recently given examples of problems where it is
possible to understand the limit $g\ra +\infty$ for fixed $q$.
\par
In this situation, by analogy with Section~\ref{ssec-arith}, one may expect 
to have a mod-Gaussian limit theorem with an extra ``arithmetic'' factor. 
More precisely, one can make the following conjecture:

\begin{conjecture}
For fixed odd $q$, we have
$$
\lim_{g\ra +\infty}{
\frac{g^{u^2/2-iu/2}}{|\tilde{\mathcal{H}}_{g,q}|}
\sum_{f\in \tilde{\mathcal{H}}_{g,q}}{
|L(f,1/2)|^{iu}
}
}=M_{Sp}(iu)A_h(iu)
$$
locally uniformly for $u\in\Rr$ and also
$$
\lim_{g\ra +\infty}{
\frac{g^{-\lambda(\lambda+1)/2}}{|\tilde{\mathcal{H}}_{g,q}|}
\sum_{f\in \tilde{\mathcal{H}}_{g,q}}{
|L(f,1/2)|^{\lambda}
}}
=M_{Sp}(\lambda)A_h(\lambda)
$$
for all real $\lambda>0$, where
\begin{multline*}
A_h(\lambda)=
\prod_{\pi}{
\Bigl(
1-\frac{1}{|\pi|}
\Bigr)^{\lambda(\lambda+1)/2}}\times\\
\Bigl(1+
\frac{1}{2}\frac{1}{1+|\pi|^{-1}}
\Bigl\{
\Bigl(1-\frac{1}{\sqrt{|\pi|}}\Bigr)^{-\lambda}-1
+
\Bigl(1+\frac{1}{\sqrt{|\pi|}}\Bigr)^{-\lambda}-1
\Bigr\}
\Bigr)
,
\end{multline*}
the product extending over all irreducible monic polynomials $\pi\in
\Fp_q[T]$, and $|\pi|=q^{\deg(\pi)}$ denoting the norm of $\pi$.
\end{conjecture}

The shape of arithmetic factors is motivated by the fact (essentially
observed by Kurlberg and Rudnick) that for a fixed $\pi$, the
$\pi$-factor of the Euler product representing $L(f,\demi)$ converges
in law to
$$
(1-T_{\pi}|\pi|^{-1/2})^{-1}
$$
where $T_{\pi}$ is a trinomial random variable taking value $0$ with
probability $\frac{1}{1+|\pi|}$ and values $\pm 1$
with equal probability $\frac{1}{2}\Bigl(1-\frac{1}{|\pi|+1}\Bigr)$; 
the inner parenthesis on the right of the arithmetic factor is, of
course, equal to $\expect((1-T_{\pi}|\pi|^{-1/2})^{-\lambda})$ (as one
sees after rearrangement of the latter).
\end{rem}

\subsection{The number of prime divisors of an integer}
\label{ssec-poisson}

A classical result of Erd\H{o}s and K\'ac states (as a particular case)
that the arithmetic function $\omega(n)$, the number of (distinct)
prime divisors of a positive integer $n\geq 1$, behaves for large $n$
like a Gaussian random variable with mean $\log \log n$ and variance
$\log\log n$, in the sense that
\begin{equation}\label{eq-erdos-kac}
\lim_{N\ra +\infty}{
\frac{1}{N}
|\{
n\leq N\,\mid\, 
a<\frac{\omega(n)-\log\log N}{\sqrt{\log\log N}}<b
\}|}
=
\frac{1}{\sqrt{2\pi}}\int_{a}^b{
e^{-t^2/2}dt}
\end{equation}
for any real numbers $a<b$. 
\par
This phenomenon where increasing variance is observed suggests, in the
context of this paper, to look at the behavior of $\omega(n)$ over
$n\leq N$ without normalizing.  However, mod-Gaussian behavior is
excluded here because of the following easy remark:

\begin{prop}
  Let $(Z_N)$ be a sequence of \emph{integer valued} random
  variables. Then $(Z_N)$ does not converges in the mod-Gaussian sense
  with respect to any unbounded parameters $(\gamma_N)$.
\end{prop}

\begin{proof}
The point is that the characteristic functions
$$
\expect(e^{iuZ_N})
$$
are $2\pi$-periodic for all $N$ if the $Z_N$ take only integral
values. If $u=2\pi$ (or any other non-zero integral multiple of
$2\pi$) then the limit~(\ref{M}) implies
$$
\lim_{N\ra +\infty}{e^{\gamma_N/2}}=|\Phi(1)|,
$$
and (since $\gamma_N\geq 0$) the existence of this limit implies that
$\gamma_N$ converges to $\gamma\geq 0$.
\end{proof}

However, it turns out that there is mod-Poisson convergence, in the
sense sketched in Section~\ref{sec-questions}. For this, it seems
slightly more appropriate to consider
\begin{equation}\label{eq-omegap}
\omega'(n)=\omega(n)-1
\end{equation}
for $n\geq 2$, because Poisson random variables takes all integral
values $\geq 0$, whereas $\omega(n)\geq 1$ for any $n\geq 2$ (of
course,~(\ref{eq-erdos-kac}) is valid for $\omega'(n)$ also).
\par
To state the result precisely, recall that a Poisson random variable
$P_{\lambda}$ with parameter $\lambda>0$ is one taking (almost surely)
integer values $k\geq 0$ with
$$
\proba(P_{\lambda}=k)=\frac{\lambda^k}{k!} e^{-\lambda}.
$$
\par
The characteristic function is then given by
$$
\E(e^{iuP_{\lambda}})=\exp(\lambda(e^{iu}-1)),
$$
and strong mod-Poisson convergence of a sequence $Z_N$ of random
variables with parameters $\lambda_N$ means that the limit
$$
\lim_{N\ra +\infty}
\exp(\lambda_N(1-e^{iu}))\E(e^{iuZ_N})=\Phi(u)
$$
exists for every $u\in\RR$, and the convergence is locally
uniform. The \emph{limiting function} $\Phi$ is then continuous and
$\Phi(0)=1$.

\begin{prop}\label{prop-mod-poisson}
  For $u\in\Rr$, let
\begin{equation}\label{eq-phi-omega}
  \Phi(u)=\frac{1}{\Gamma(e^{iu}+1)}\prod_p{
  \Bigl(1-\frac{1}{p}\Bigr)^{e^{iu}}
  \Bigl(1+\frac{e^{iu}}{p-1}\Bigr)
}.
\end{equation}
\par
This Euler product is absolutely and locally uniformly
convergent. Moreover, for any $u\in \RR$, we have
$$
\lim_{N\ra +\infty}{ \frac{ (\log N)^{(1-e^{iu})}}{N}\sum_{2\leq n\leq
    N}{e^{iu\omega'(n)}}}=\Phi(u),
$$
and the convergence is locally uniform. 
\end{prop}

\begin{proof}
Since
$$
\frac{1}{N-1}\sum_{2\leq n\leq N}{e^{iu\omega'(n)}}= e^{-iu}\times
\frac{1}{N}\sum_{1\leq n\leq N}{e^{iu\omega(n)}} +O(1),
$$
for $N\geq 2$, this is in fact a simple reinterpretation of a direct
application of the Delange-Selberg method (see, e.g.,~\cite[II.5,
Theorem 3]{tenenbaum}) to the multiplicative function $n\mapsto
e^{iu\omega(n)}$. The details are explained in~\cite[II.6, Theorem
1]{tenenbaum} (take $N=0$, $z=e^{iu}$, $A=1$ there, and apply
$\lambda_0(z)=1$ to get the formula
  $$
  \frac{1}{\Gamma(e^{iu})}G_1(1;e^{iu})=
  \frac{1}{\Gamma(e^{iu})}\prod_p{ \Bigl(1-\frac{1}{p}\Bigr)^{e^{iu}}
    \Bigl(1+\frac{e^{iu}}{p-1}\Bigr) }
$$ 
for the limit, which is in the notation of loc. cit., with $G_1(s;z)$
defined on the last line of p. 201 of~\cite{tenenbaum}, its analytic
continuation to $\Reel(s)>1/2$ being described on p. 202). Multiplying
by $e^{-iu}$, we obtain the stated result using
$e^{iu}\Gamma(e^{iu})=\Gamma(e^{iu}+1)$.
\end{proof}

\begin{cor}\label{cor-omega-1}
Consider random variables $M_N$, for $N\geq 2$, such that
$$
\proba(M_N=n)=\frac{1}{N-1},\quad\quad 2\leq n\leq N,
$$
and let $Z_N=\omega'(M_N)$. Then the sequence $(Z_N)$ converges
strongly in the mod-Poisson sense with limiting function $\Phi$ given
by~\emph(\ref{eq-phi-omega}) and parameters $\lambda_N=\log\log N$.
\end{cor}

\begin{proof}
  This follows directly from the proposition and the definition of
  strong mod-Poisson convergence, since we have
  $\expect(e^{iuP_{\lambda_N}})^{-1}= (\log N)^{(1-e^{iu})}$.
\end{proof}

The analogue of Corollary~\ref{cor-clt} is then the theorem
of Erd\H{o}s and K\'ac:

\begin{cor}
The Gaussian limit~\emph{(\ref{eq-erdos-kac})} holds.
\end{cor}

\begin{proof}
  With notation as in Corollary~\ref{cor-omega-1}, we must show that
$$
Y_N=\frac{\omega(M_N)-\log\log N}{\sqrt{\log\log N}}
$$
converges in law to a standard Gaussian variable (with a shift from
$N$ to $N+1$ which is immaterial). The argument is again quite
standard, but we spell it out in detail.
\par
Let $u\in\RR$ be fixed; the characteristic function of $Y_N$ is
\begin{align}\label{eq-charfn}
  \E(e^{iuY_N})&=\exp(-iu\sqrt{\log\log N})\E(e^{it\omega(M_N)})\nonumber \\
  &= \exp(-iu\sqrt{\log\log N}+it)\E(e^{it\omega'(M_N)})
\end{align}
where
$$
t=\frac{u}{\sqrt{\log\log N}}
$$
(note that $t$ depends on $N$ and $t\ra 0$ when $N\ra +\infty$).
\par
By Proposition~\ref{prop-mod-poisson}, in particular the uniform
convergence with respect to $u$, we have
\begin{equation}\label{eq-1}
\lim_{N\ra +\infty}{(\log
  N)^{1-e^{it}}\E(e^{it\omega'(M_N)})}=\Phi(0)=1.
\end{equation}
\par
Moreover, we have for $N\geq 1$
\begin{align}
(\log N)^{e^{it}-1}&=
\exp((e^{it}-1)\log\log N)\nonumber\\
&=\exp((it-t^2/2+O(t^3))\log\log N)\nonumber\\
&=\exp\Bigl(iu\sqrt{\log\log N}-\frac{u^2}{2}+
O\Bigl(
\frac{u^3}{\sqrt{\log\log N}}
\Bigr)\Bigr).
\label{eq-2}
\end{align}
\par
Writing~(\ref{eq-charfn}) as
$$
\exp(-iu\sqrt{\log\log N}+it) \times
(\log N)^{e^{it}-1}\times (\log
N)^{1-e^{it}}\E(e^{it\omega'(M_N)}) ,
$$
we see from~(\ref{eq-1}) and~(\ref{eq-2}) that this is
$$
\exp\Bigl(-\frac{u^2}{2}+O\Bigl(\frac{u^3}{\sqrt{\log\log N}}\Bigr)
\Bigr)
(1+o(1))\ra \exp\Bigl(-\frac{u^2}{2}\Bigr),\quad \text{ as }\quad
N\ra +\infty,
$$
and by L\'evy's criterion, this concludes the proof.
\end{proof}

In fact, this proof is essentially the one of R\'enyi and
Tur\'an~\cite{renyi-turan}, who simply did not isolate
Proposition~\ref{prop-mod-poisson} as a separate statement of interest
(and proved directly the version of the Delange-Selberg needed in this
case, see equation (1.31) in loc. cit.).

\begin{rem}
It is also natural to use Poisson variables because the asymptotic
formula 
$$
\frac{1}{N} |\{n\leq N\,\mid\, \omega'(n)=k\}|\sim \frac{1}{\log
  N}\frac{(\log\log N)^{k}}{k!}
$$
holds as $N\ra +\infty$, for fixed $k\geq 0$ (the uniformity with
respect to $k$, as shown first by Sathe and Selberg, is a quite
delicate issue, see, e.g.,~\cite[II.6, Theorem 4]{tenenbaum}), so that
$\omega'(n)$ is again seen to be ``approximately'' Poisson with
parameter $\log\log N$.
\end{rem}

Our focus here is on arithmetic behavior, but it is interesting to
note a phenomenon similar (though apparently simpler) to what happens
for the Moment Conjecture for $\zeta(1/2+it)$, namely, that the
limiting function $\Phi(u)$ takes the form of a product
$$
\Phi(u)=\Phi_1(u)\Phi_2(u)
$$
with
$$
\Phi_1(u)=\frac{1}{\Gamma(e^{iu}+1)},\quad\quad \Phi_2(u)= \prod_{p}{
  \Bigl(1-\frac{1}{p}\Bigr)^{e^{iu}-1} \Bigl(1-\frac{1}{p}\Bigr)
  \Bigl(1+\frac{e^{iu}}{p-1}\Bigr) },
$$
and we can see the Euler product $\Phi_2$ as an instance of a natural
asymptotic probabilistic model of primes, while $\Phi_1$ comes from a
seemingly unrelated model which has some group-theoretic origin.
\par
Indeed, for the first factor $\Phi_1(u)=\Gamma(e^{iu}+1)^{-1}$,
appealing to the classical formula
$$
\frac{1}{\Gamma(e^{iu}+1)}=\prod_{k\geq 1}{
  \Bigl(1+\frac{e^{iu}}{k}\Bigr)
  \Bigl(1+\frac{1}{k}\Bigr)^{-e^{iu}}
}
$$
for any $u\in\Rr$ (due to Euler; see~\cite[12.11]{ww}), we can compute
this as follows:
\begin{align*}
  \Phi_1(u)&=\lim_{N\ra +\infty}{ \prod_{k\leq N}{
      \Bigl(1+\frac{1}{k}\Bigr)^{1-e^{iu}}
      \Bigl(1+\frac{1}{k}\Bigr)^{-1} \Bigl(1+\frac{e^{iu}}{k}\Bigr) }
  }\\
  &= \lim_{N\ra +\infty}\exp(\lambda_N(1-e^{iu}))\prod_{k\leq N}{
    \Bigl(1+\frac{1}{k}\Bigr)^{-1} \Bigl(1+\frac{e^{iu}}{k}\Bigr)
  }\\
  &=\lim_{N\ra +\infty}\exp(\lambda_N(1-e^{iu})) \expect(e^{iu Z_N}),
\end{align*}
where 
$$
\lambda_N=\sum_{1\leq k\leq N}{\log (1+k^{-1})}=\log (N+1),
$$
and $Z_N$ is the sum 
$$
Z_N=B_{1}+B_{2}+\cdots +B_{N},
$$
with $B_{k}$ denoting independent Bernoulli random variables with
distribution
$$
\proba(B_{k}=1)=1-\frac{1}{1+\frac{1}{k}}=\frac{1}{k+1},\quad\quad
\proba(B_{k}=0)=\frac{1}{1+\frac{1}{k}}=\frac{k}{k+1}.
$$
\par
This distribution is also found, for instance, as the distribution of
$K_{N+1}-1$, where $K_{N+1}$ is the number of cycles of a permutation
$\sigma\in \mathfrak{S}_N$, distributed according to the uniform
measure on the symmetric group (see, e.g.,~\cite[\S 4.6]{abt}; we
observe in particular that this gives another instance of natural
mod-Poisson convergence explaining some type of normal limit). So here
we see random permutations occuring as explaining the
``transcendental'' factor $\Phi_1(u)$; in~\cite{kn}, this phenomenon
is explored in greater depth.
\par
For the arithmetic factor $\Phi_2(u)$, we argue (much as in
Section~\ref{ssec-arith}) that
$$
\Phi_2(u)=\lim_{y\ra +\infty}{\prod_{p\leq y}{
    \Bigl(1-\frac{1}{p}\Bigr)^{e^{iu}-1}
\Bigl(1-\frac{1}{p}\Bigr)
\Bigl(1+\frac{e^{iu}}{p-1}\Bigr)
}},
$$
and by isolating the first term, it follows that
\begin{align*}
\Phi_2(u)&=\lim_{y\ra +\infty}{
\exp((1-e^{iu})\lambda_y)
\prod_{p\leq y}{\Bigl(1-\frac{1}{p}+\frac{1}{p}e^{iu}\Bigr)
}}\\
&=\lim_{y\ra +\infty}{\E(e^{iuP_{\lambda_y}})^{-1}}
\E(e^{iuZ'_y})
\end{align*}
where
$$
\lambda_y=\sum_{p\leq y}{\log (1-p^{-1})^{-1}}=
\sum_{\stacksum{p\leq y}{k\geq 1}}{
\frac{1}{kp^k}}\sim \log\log y,\quad\text{ as } y\ra +\infty,
$$
and
\begin{equation}\label{eq-zy}
Z'_y=\sum_{p\leq y}{B'_{p}}
\end{equation}
is a sum of independent Bernoulli random variables with parameter
$1/p$:
$$
\proba(B'_{p}=1)=\frac{1}{p},\quad\quad
\proba(B'_{p}=0)=1-\frac{1}{p}.
$$
\par
The parameters of these Bernoulli laws correspond exactly to the
``intuitive'' probability that an integer $n$ be divisible by $p$, or
in other words to the model limit
$$
\lim_{x\ra +\infty}{\frac{1}{x}\sum_{\stacksum{n\leq x}{n\equiv
      0\mods{p}}}{1}}
=\frac{1}{p},
$$
and the independence of the $B'_p$ corresponds to the formal
(algebraic) independence of the divisiblity by distinct primes given,
e.g., by the Chinese Remainder Theorem.
\par
Since the analogue of Proposition~\ref{prop-formal}, (3), is trivially
valid for mod-Poisson convergence, we also recover in this manner
(without arithmetic) the fact that the limiting function $\Phi(u)$
arises from mod-Poisson convergence. 
\par
As in the case of the Riemann zeta function, we note that the
independent model fails to capture the truth on the distribution of
$\omega(n)$, the extent of this failure being given by the factor
$\Phi_1(u)$. Because
$$
\frac{Z'_y-\log \log y}{\sqrt{\log\log y}}\fleche{law} \mathcal{N}(0,1)
$$
(with the right-hand side being a standard normal random variable),
this discrepancy between the independent model and the arithmetic
truth is invisible at the level of the normalized convergence in
distribution. 
\par

\begin{rem}
  Computations similar to the above show that, for any sequence
  $(x_n)$ of positive real numbers with
\begin{equation}\label{eq-mod-poisson-cond}
\sum_{n\geq 1}{x_n}=+\infty,\quad\quad \sum_{n\geq 1}{x_n^2}<+\infty, 
\end{equation}
if we denote now by $(B_n)$ a sequence of independent Bernoulli random
variables with
$$
\proba(B_n=1)=x_n,\quad\quad \proba(B_n=0)=1-x_n,
$$
then the random variables
$$
Z_N=B_1+\cdots+B_N
$$
have mod-Poisson convergence with parameters
$$
\lambda_N=x_1+\cdots +x_N,
$$
and with limiting function given by
$$
u\mapsto \prod_{n\geq 1}{(1+x_n(e^{iu}-1))\exp(x_n(1-e^{iu}))}\ ;
$$
the (uniform) convergence of this infinite product is ensured by the
second condition in~(\ref{eq-mod-poisson-cond}), after expanding in
terms of $x_n$ (which tends to $0$ as $n\ra +\infty$).
\end{rem}


\renewcommand{\refname}{References}

\end{document}